\documentclass[reqno,12pt]{amsart}
\usepackage[colorlinks,pagebackref,hypertexnames=false]{hyperref}
\usepackage{fancybox,fancyhdr,graphics,epsfig}

\usepackage{amsmath,amssymb,mathrsfs,mathtools,verbatim} 
\usepackage[alphabetic,backrefs]{amsrefs}
\usepackage{ae,aecompl}
\usepackage[margin=1in]{geometry}
\usepackage[onehalfspacing]{setspace}
\usepackage{microtype}
\usepackage{tikz}
\usepackage{dsfont}
\usepackage{xcolor}

\hypersetup{
	colorlinks,
	linkcolor={red!45!black},
	citecolor={blue!45!black},
	urlcolor={magenta}
}

\usepackage[backrefs]{amsrefs}

\usepackage[english]{babel}

\usepackage{bookmark}

\numberwithin{equation}{section}

\newcommand{\rG}{{\rm G}}

\newcommand{\cC}{\mathcal{C}}

\newcommand{\SO}{{\rm SO}}

\newcommand{\SU}{{\rm SU}}
\newcommand{\GL}{\mathrm{GL}}

\newcommand{\U}{{\rm U}}

\newcommand{\Ad}{\mathrm{Ad}}

\renewcommand{\epsilon}{\varepsilon}

\newcommand{\Hom}{{\mathrm{Hom}}}

\newcommand{\delbar}{\bar \del}
\newcommand{\del}{\partial}
\newcommand{\diag}{\mathrm{diag}}

\newcommand{\dvol}{\mathop\mathrm{dvol}\nolimits}
\newcommand{\id}{\mathrm{id}}

\renewcommand{\Im}{\mathop{\mathrm{Im}}}

\renewcommand{\Re}{\mathop{\mathrm{Re}}}

\def\<{\mathopen{}\left<}
\def\>{\right>\mathclose{}}
\def\({\mathopen{}\left(}
\def\){\right)\mathclose{}}

\newtheorem{theorem}{Theorem}

\newtheorem{corollary}{Corollary}

\newtheorem{definition}{Definition}
\newtheorem{example}{Example}

\newtheorem{proposition}{Proposition}
\newtheorem{remark}{Remark}

\numberwithin{equation}{section}

\title{The DT-instanton equation on almost Hermitian $6$-manifolds}

\author{Gavin Ball}
\address{\textsc{Universit\'{e} du Qu\'{e}bec \`{a} Montr\'{e}al}, 	\textsc{D\'{e}partement de math\'{e}matiques}, \textsc{Case postale 8888, succursale centre-ville, Montr\'{e}al (Qu\'{e}bec), H3C 3P8, Canada}}
\email{gavin.ball@cirget.ca}
\urladdr{https://www.gavincfball.com/} % Delete if not wanted.
\author{Gon\c{c}alo Oliveira}
\address{Universidade Federal Fluminense IME--GMA, Niter\'oi, Brazil}
\email{galato97@gmail.com}

\begin{document}
\maketitle

%===============================================================================
\begin{abstract}
	This article investigates a set of partial differential equations, the DT-instanton equations, whose solutions can be regarded as a generalization of the notion of Hermitian-Yang-Mills connections. These equations owe their name to the hope that they may be useful in extending the DT-invariant to the case of symplectic $6$-manifolds.
	
	In this article, we give the first examples of non-Abelian and irreducible DT-instantons on non-K\"ahler manifolds. These are constructed for all homogeneous almost Hermitian structures on the manifold of full flags in $\mathbb{C}^3$. Together with the existence result we derive a very explicit classification of homogeneous DT-instantons for such structures. Using this classification we are able to observe phenomena where, by varying the underlying almost Hermitian structure, an irreducible DT-instanton becomes reducible and then disappears. This is a non-K\"ahler analogue of passing a stability wall, which in string theory can be interpreted as supersymmetry breaking by internal gauge fields.
\end{abstract}
%===============================================================================

%===============================================================================

\tableofcontents

\section{Introduction}

\subsection{Summary}

The notions of holomorphic bundles and Hermitian Yang Mills connections have proven to be very fruitful in complex geometry. When considering Hermitian vector bundles, the Hitchin-Kobayashi correspondence \cite{D,UY,LT} yields a relation between the algebro-geometric notion of a stable holomorphic vector bundle and the more differential geometric one of a Hermitian-Yang-Mills (HYM) connection. The goal of the present paper is to study some natural generalizations of these objects in almost complex geometry. The most well known of such generalizations are pseudo-holomorphic and pseudo-Hermitian-Yang-Mills (pHYM) connections. In specific situations, these have been studied by several authors, see \cite{Charbonneau2016} and \cite{Bryant2006} for example. The major goal of the current paper is to study a system of partial differential equations whose solutions give a further generalization of the notion of a HYM connection on real $6$-dimensional almost Hermitian manifolds. To the authors' knowledge such equations first appeared in Richard Thomas's thesis \cite{Thomas1997} (page 29). These equations have also independently appeared in the physics literature, for instance in \cite{Baulieu1998,Baulieu1998b,Iqbal2008} and references therein. More recently, the same equations were studied by Yuuji Tanaka (\cite{Tanaka2008}, \cite{Tanaka2013}, \cite{Tanaka2014}) who constructed the only known (nontrivial) examples of solutions in \cite{Tanaka2008}. These rely on a very general version of the Hitchin-Kobayashi correspondence and require the underlying almost Hermitian manifold to actually be K\"ahler. In that direction, our results give the first nontrivial solutions to these equations on non-K\"ahler almost Hermitian manifolds. For instance, one of the examples explored in this paper focuses on $\mathbb{F}_2$, the manifold of full flags in $\mathbb{C}^3$. In that example, we give nontrivial solutions to these equations for several almost Hermitian structures compatible with the nearly K\"ahler almost complex structure.

\subsection{The DT-instanton equations}

Let $(X,g,J)$ be an almost Hermitian manifold, $\rG$ a compact semisimple\footnote{Similar equations to those considered here can be written if $\rG$ not semisimple. However, for the sake of simplifying some statements we shall restrict to this case.} Lie group, and $P \rightarrow X$ a principal $\rG$-bundle. A connection $A$ on $P$ is called pseudo-holomorphic if its curvature $F_A$ is of type $(1,1)$ and pHYM if we further have
\begin{equation}\label{eq:HYM}
\Lambda F_A = 0,
\end{equation}
where $\Lambda F_A = \ast (F_A \wedge \omega^2)$ denotes contraction with respect to the associated $2$-form $\omega(\cdot , \cdot)= g(J \cdot , \cdot)$. These notions are word by word adaptations of the respective notions in the case where $(X,g,J)$ is Hermitian. However, for the general almost Hermitian structure there may not exist (even locally) solutions to these equations, see \cite{Bryant2006}. Similarly, in the non-integrable case there is no analogue of the function theory relating the existence of a (pseudo)-holomorphic connection with local holomorphic framings. When $X$ is real $6$-dimensional, there is a further generalization of the HYM equations which has a better general existence theory. This is an equation for a pair $(A, u)$, consisting of a connection $A$ on $P$ and a Higgs field $u \in \Omega^{0,3}(X ,\mathfrak{g}_P^{\mathbb{C}})$, required to satisfy
\begin{eqnarray}\label{eq:New_Mon1}
F_A^{0,2} & = &  \overline{\partial}_A^* u, \\ \label{eq:New_Mon2}
\Lambda F_A & = &  \ast [ u \wedge \overline{ u} ],
\end{eqnarray}
where $\overline{\partial}_A^* = -\ast \partial_A \ast$ and $\ast$ denotes the $\mathbb{C}$-linear extension of the Hodge-$\ast$ operator associated with $g$. In this paper we shall refer to these as the \emph{DT-instanton equations}, and call a pair $(A,u)$ solving them a \emph{DT-instanton}. The reason for this nomenclature is the point of view adopted in Richard Thomas and Yuuji Tanaka's work, where these equations are regarded as the basis for a possible differential geometric approach to the Donaldson-Thomas invariants, constructed by Thomas for Calabi-Yau $3$-folds using algebraic geometry, see \cite{Thomas1997} and \cite{Thomas2001}. Such a program is still far from being completed, but its success would extend the theory of DT-invariants to symplectic (or even almost Hermitian) real $6$-dimensional manifolds.

As it will become clear in the course of the article, a particular case for these equations happens when $(J,g)$ admits a certain compatible $\SU(3)$-structure. In that situation, the equations can be rewritten as in proposition \ref{prop:New_Mon_Eq}, leading to a simplification of the analysis which can be carried out in several particular cases of interest. For instance, when the $\SU(3)$-structure is Calabi-Yau or nearly K\"ahler, the DT-instanton equations actually reduce to the pHYM ones, see proposition \ref{prop:NKMon}. These vanishing theorems further motivate the DT-instanton equations as being a natural generalization of the HYM ones which, for the generic almost Hermitian structure, we expect to have more solutions then the pHYM equations. To test these ideas, in this paper we solve these equations in very specific examples, namely for invariant almost Hermitian structures on the manifold of full flags in $\mathbb{C}^3$. In particular, we obtain the first examples of DT-instantons on a compact manifold with $\overline{\partial}_A^* u \neq 0$, and these are also the first non-trivial examples on non-K\"ahler manifolds. For future reference, DT-instantons $(A,u)$ with $A$ an irreducible connection and $\overline{\partial}_A^* u \neq 0$ will be called \emph{irreducible}.

\subsection{Main results}

\noindent As mentioned before, the specific examples we shall study focus on $\mathbb{F}_2$, the manifold of full flags in $\mathbb{C}^3$. This is an homogeneous space of the form $\SU(3)/T^2$ and we consider $\SU(3)$-invariant almost Hermitian structures on it. The space of such almost Hermitian structures, which we shall denote by $\mathcal{C}$, has two connected components $\cC^i$, $\cC^{ni}$, both of which can be identified with $\mathbb{R}^3$. These components respectively correspond to those almost Hermitian structures which are compatible with the standard integrable almost complex structure $J^i$ and to one other non-integrable almost complex structure $J^{ni}$ which is in fact compatible with a nearly K\"ahler structure. Indeed, $\mathbb{F}_2$ admits two homogeneous Einstein metrics: the K\"ahler-Einstein one $g^{ke}$ with $(g^{ke},J^{i}) \in \cC^i$; and another one $g^{nk}$ so that $(g^{nk}, J^{ni}) \in \cC^{ni}$ is nearly K\"ahler. For future reference, given $(g,J) \in \cC$ we shall denote by $\omega$ the associated fundamental $2$-form and we note here that for all $(g,J) \in \cC$ the $4$-form $\omega^2$ is exact and thus gives an associated cohomology class on $\mathbb{F}_2$.

In fact, of these two almost complex structures, $J^{ni}$ is the only one with topologically trivial canonical bundle. As a consequence of this observation, for the almost Hermitian structures in $\mathcal{C}^{ni}$ we will be able to simplify the search for solutions to the DT-instanton equations \ref{eq:New_Mon1}--\ref{eq:New_Mon2} by making use of proposition \ref{prop:New_Mon_Eq}. Indeed, in this case we are able to classify $\SU(3)$-invariant DT-instantons with gauge group $\SO(3)$. To state this classification we start with some preparation. Let $r \in (\mathfrak{t}^2)^*$ be an integral weight of $\SU(3)$ and $e^{r}:T^2 \rightarrow \U(1)$ the induced group homomorphism. Then, denote by $\lambda_{r}:T^2 \rightarrow \SO(3)$ the homomorphism obtained by composing $e^{r}$ with the degree one embedding of $\U(1)$ as the maximal torus of $\SO(3)$. This can be used to construct the $\SU(3)$-homogeneous $\SO(3)$-bundles
$$P_r = \SU(3) \times_{(T^2, \lambda_r)} \SO(3).$$
All the $\SU(3)$-homogeneous $\SO(3)$-bundles on $\mathbb{F}_2$ are of this form, and in our first main result we classify $\SU(3)$-invariant DT-instantons on such bundles. As way of preparing for the statement it is important to note that the bundles $P_r$ are all reducible to the $\U(1)$-bundles $L_r$ associated with the homomorphism $e^{r}$. The Hitchin-Kobayashi correspondence, and also our analysis, yields that for $(g,J^i) \in \cC^i$ an irreducible HYM connection on the bundle $P_r$ exists if and only if $\deg(L_r) <0$. In other words, for invariant Hermitian structures $(g,J^i) \in \cC^i$, the quantity $\deg(L_{r})$ controls the existence of HYM connections on the bundle $P_r$. Here, by $\deg(L_{\beta})$, i.e. the degree of $L_{\beta}$, we mean the value of $c_1(L_{\beta}) \cup [\omega^2]$ evaluated against the fundamental class of $(\mathbb{F}_2,J)$, which we regard as a real valued function on both $\cC^i$ and $\cC^{ni}$. Our main result, stated below, shows that, for invariant almost Hermitian structures $(g,J^{ni}) \in \cC^{ni}$, the same quantity controls the existence of invariant DT-instantons on $P_r$.

\begin{theorem}\label{thm:Intro_F2_Main}
	Let $(g, J^{ni}) \in \mathcal{C}^{ni}$ and $(A,u)$ be a $\SU(3)$-invariant DT-instanton on $P_r$. Then, $r$ must be a root of $\SU(3)$ and 
	$$\deg(L_r) \leq 0.$$ 
	Moreover, the resulting DT-instanton is irreducible if and only if strict inequality holds.
\end{theorem}

The previous result, i.e. Theorem \ref{thm:Intro_F2_Main}, is a summarized version of our main result stated as Theorem \ref{thm:Monopoles_Flag_JnK} which constructs these DT-instantons and the functions $\deg(L_r)$ explicitly. In particular, this gives the first existence theorem for solutions of the equations \ref{eq:New_Mon1}--\ref{eq:New_Mon2} with $\overline{\partial}_A^* u \neq 0$, which are also the first nontrivial examples outside the K\"ahler world. 

We must further mention on the appearance of the restriction that the integral weights $r$ are roots of $\SU(3)$ follows in our case from the imposition that the DT-instantons $(A,u)$ be $\SU(3)$-invariant. In fact, we do not expect that this condition must hold for general DT-instantons, but of course we do not know if these exist on other $P_r$.

An immediate consequence of the explicit nature of our formulas for the degrees $\deg(L_r)$ as functions on $\cC^{ni}$ is that it is very easy to check that they cannot all be simultaneously negative. Furthermore, they all vanish if and only if $g=g^{nk}$ is the metric compatible with the nearly K\"ahler structure. This gives the result below, which will be presented in a more detailed manner as corollary \ref{cor:Monopoles_Jnk_Flag}.

\begin{corollary}
	Let $(g, J^{ni}) \in \mathcal{C}^{ni}$ be a $\SU(3)$-invariant almost Hermitian structure compatible with $J^{ni}$. Then, either:
	\begin{itemize}
		\item[(i)] There is a root $r$ together with an irreducible $\SU(3)$-invariant DT-instanton on $P_r$, or
		\item[(ii)] $g=g^{nk}$ is the nearly-K\"ahler metric and there is a reducible pHYM connection on all the bundles $P_r$. In this case, the corresponding connection is reducible to $L_r \hookrightarrow P_r$.
	\end{itemize}
\end{corollary} 

In fact, one can explicitly pick a $1$-parameter family of compatible $2$-forms $\lbrace \omega_s \rbrace_{s \in I \subset \mathbb{R}}$ so that for some root $\deg(L_r)(\omega_s)$ crosses zero. Then, one can check that as $\deg(L_r)(\omega_s) \searrow 0$, the DT-instanton $(A,u)$ constructed by theorem \ref{thm:Intro_F2_Main} become obstructed and reducible. See examples \ref{ex:Mon_1}--\ref{ex:Mon_2} and the accompanying figures for illustrations of this phenomena. From a Physics point of view, this phenomenon can be interpreted as analogous to crossing a wall from the supersymmetric region in $\cC^{ni}$ to a non-supersymmetric one. Namely, it can be thought of as analogous to supersymmetric breaking by internal gauge fields. See \cite{Anderson2009} for the description of this phenomenon in the K\"ahler case.

In section \ref{sect:DTSO3} we turn to the problem of classifying the $\SU(3)$-invariant pHYM connections with gauge group $\SO(3)$. In theorem \ref{thm:HYM_Integrable_Flag} we prove that the existence of such irreducible pHYM connections requires the almost complex structure to be integrable and we write down the resulting HYM connections explicitly.

\subsubsection*{Acknowledgements}

We would like to thank Benoit Charbonneau, G\"ael Cousin and Lorenzo Foscolo for helpful conversations regarding this article. We are particularly thankful for Benoit Charbonneau's comments and carefully reading a previous version of this article.\\
Gon\c{c}alo Oliveira is supported by Funda\c{c}\~ao Serrapilheira 1812-27395, by CNPq grants 428959/2018-0 and 307475/2018-2, and FAPERJ through the program Jovem Cientista do Nosso Estado E-26/202.793/2019.

\section{Preliminaries}

\subsection{Almost Hermitian 6-manifolds and $\SU(3)$-structures}

An almost Hermitian $6$-manifold is a triple $(X^6,J,g)$, where $X$ is real $6$-dimensional smooth manifold equipped with an almost complex structure $J$ and a compatible Riemannian metric $g$, i.e $g ( J \cdot , J \cdot)= g ( \cdot , \cdot)$. In this situation, the associated $(1,1)$-form is given by $\omega ( \cdot , \cdot ) = g ( J \cdot , \cdot)$. The pair $(J,g)$ determines a reduction of the structure group of the frame bundle of $X$ from $\GL(6, \mathbb{R})$ to $\U(3)$. Thus, when convenient, we refer to the pair $(J,g)$ as an $\U(3)$-structure. In this manner, a $\SU(3)$-structure compatible with a given $(J,g)$ consists of the extra data of a real $3$-form $\Omega_1$ satisfying
$$\omega \wedge \Omega_1 =0 , \quad \omega^3 = - \frac{3}{2} \Omega_1 \wedge \Omega_2 , $$
where $\Omega_2=J\Omega_1= \ast \Omega_1$. In particular, the complex valued $3$-form $\Omega = \Omega_1 + i \Omega_2$ is of type $(3,0)$ with respect to $J$. Indeed, such a triple $(J,g,\Omega_1)$ determines a reduction of the structure group of the frame bundle to $\SU(3)$. In order to settle on the nomenclature we now recall various different kinds of $\SU(3)$-structures we will be considering

\begin{definition}\label{def:SU(3)_Structures}
	A $\SU(3)$-structure $(J,g,\Omega_1)$, or equivalently $(J,\omega,\Omega_1)$, will be called
\begin{enumerate}
\item[(a)] Calabi-Yau if $d \omega = 0 = d (\Omega_1 + i \Omega_2) $;
\item[(b)] Nearly Calabi-Yau if $d \omega=0=d \Omega_2$;
\item[(c)] Nearly K\"ahler if $d \omega = 3 \Omega_1 $, and $d \Omega_2 = -2 \omega^2$;
\item[(d)] Half-flat if $d \omega^2 =0=d \Omega_1$.
\end{enumerate}
\end{definition} Notice that a Calabi-Yau structure is also nearly Calabi-Yau and half-flat, while a nearly K\"ahler structure is also half-flat.

\begin{remark}
Consider the product $S^1_{\theta} \times X$ equipped with the $\mathrm{G_2}$-structure $\varphi = d\theta \wedge \omega + \Omega_2$ (and the orientation such that $\psi=- d \theta \wedge \Omega_1 + \frac{1}{2} \omega^2$. This $\mathrm{G_2}$-structure is closed if and only if the $\SU(3)$-structure is nearly Calabi-Yau, and coclosed if and only if it is half-flat. This justifies our choice of $\Omega_2$ rather than $\Omega_1$ in the definition of nearly Calabi-Yau.
\end{remark}

\subsection{Pseudo-holomorphic and pHYM connections}

Throughout this paper $G$ will be a compact Lie group, with Lie algebra $\mathfrak{g}$, and $P$ a principal $G$-bundle over $X$. The adjoint bundle of $P$ will be denoted by $\mathfrak{g}_P$. Recall, for example from \cite{Bryant2005}, that

\begin{definition}
Let $A$ be an Hermitian connection on $P$ and $F_A$ denote its curvature. Then, $A$ is said to be pseudo-holomorphic if 
\begin{equation}\label{eq:holomorphic}
F_A^{0,2} =0,
\end{equation}
and pseudo Hermitian-Yang-Mills (pHYM), if
\begin{eqnarray}\label{eq:HYM1}
F_A^{0,2} & = & 0 \\ \label{eq:HYM2}
\Lambda F_A & = & \lambda C ,
\end{eqnarray}
where $\Lambda F_A = \ast (F_A \wedge \omega^2)$, $\lambda \in \mathbb{R}$ is a constant, and $C$ is a constant central element.
\end{definition}

When $(X,J,g)$ is an Hermitian manifold, the notions of pseudo-holomorphic and pHYM connections obviously coincide with the usual ones of holomorphic and HYM connections respectively. This motivates the definitions above in the setting of general almost Hermitian structures where much less is known. See for example \cite{Charbonneau2016,Bryant2005}, and references therein, for the case of nearly K\"ahler manifolds.

\begin{remark}\label{rem:HYM_In_Nearly_Kahler}
For a nearly K\"ahler structure $(X, \Omega_1, \omega)$ any pseudo-holomorphic connection $A$ is immediately pHYM. Indeed, the curvature $F_A$ of a pseudo holomorphic connection $A$ is of type $(1,1)$. Hence, for any compatible $\SU(3)$-structure $F_A \wedge \Omega $ vanishes. Thus, for a nearly K\"ahler structure $(\omega, \Omega_1)$ we have $F_A \wedge \Omega_2 =0$ and differentiating this equation gives $F_A \wedge \omega^2 =0$, so $A$ is also pHYM. 
\end{remark}

\subsection{DT-instantons}

The conclusion of remark \ref{rem:HYM_In_Nearly_Kahler} is a shadow of the first of the following two interesting facts about nearly K\"ahler structures, see \cite{Bryant2005,Foscolo2016,Verbitsky2011}:
\begin{itemize}
	\item[(a)] All closed $2$-forms of type $(1,1)$ are primitive, i.e satisfy $F \wedge \omega^2 =0$. This follows from the computation in remark \ref{rem:HYM_In_Nearly_Kahler} by replacing the curvature by any such $2$-form.
	
	%fact that $d \Omega_2 = - 2 \omega^2$, as follows: given a closed $F$ of type $(1,1)$, we have $F \wedge \Omega_2=0$ and this implies
	%$$0=d(F \wedge \Omega_2 )= F \wedge d \Omega_2 = - 2 F \wedge \omega^2.$$
	
	\item[(b)] All closed, $2$-forms of type $(1,1)$ are harmonic. This follows easily from the previous bullet and the fact that any primitive $(1,1)$-form satisfies $\ast F = - F \wedge \omega$, hence
	$$d \ast F = - F \wedge d \omega = -3F \wedge \Omega_1=0,$$
	as $F$ is of type $(1,1)$.
\end{itemize}

Robert Bryant identified in \cite{Bryant2005} a class of Hermitian-structures on a $6$-manifold for which closed $2$-forms of type $(1,1)$ are primitive. Bearing in mind the previous remark, it would not be surprising if they had a good local existence theory for pseudo-holomorphic and pHYM connections. Indeed, these structures are called quasi-integrable and were shown in \cite{Bryant2005} to locally admit as many pseudo-holomorphic and pHYM connections as the integrable ones.

However, for the general almost Hermitian structure, the pHYM equations are overdetermined modulo gauge. Motivated by this we shall now introduce an elliptic equation (modulo gauge) whose solutions we regard as generalizing the notion of pHYM equation to the generic almost Hermitian structure. 

\begin{definition}
A pair $(A, u)$ where $A$ is a connection on $P$ and $u \in \Omega^{0,3}(X ,\mathfrak{g}_P^{\mathbb{C}})$ is called a DT-instanton if
\begin{eqnarray}\label{eq:New_Mon1}
F_A^{0,2} & = &  \overline{\partial}_A^* u, \\ \label{eq:New_Mon2}
\Lambda F_A & = &  \ast [ u \wedge \overline{ u} ],
\end{eqnarray}
where $\overline{\partial}_A^* = -\ast \partial_A \ast$ and $\ast$ denotes the $\mathbb{C}$-linear extension of the Hodge-$\ast$ operator. A DT-instanton $(A,u)$ will be called irreducible if the connection $A$ is irreducible.
\end{definition}

To the authors' knowledge, these equations first appeared in \cite{Thomas1997} and were considered in this set-up already in \cite{Donaldson2009}. They have been studied by Donaldson-Segal \cite{Donaldson2009} and Yuuji Tanaka \cite{Tanaka2008}, \cite{Tanaka2013}, \cite{Tanaka2014} who suggest these equations as a possible analytic approach to DT-invariants. On noncompact Calabi-Yau manifolds they were studied by the second named author in \cite{Oliveira2014} and \cite{Oliveira2016}. The same equations have also been studied by physicists such as in \cite{Baulieu1998} and \cite{Iqbal2008}.

\section{DT-instantons in special cases}

In this subsection we study the DT-instanton equations in several particular cases. We give a few vanishing theorems which further motivate the view of the DT-instanton equations as a generalization of the HYM equations.

\subsection{On Hermitian manifolds}

If $X$ is compact and $J$ integrable, then any DT-instanton induces a holomorphic structure on any associated complex vector bundle. Indeed, it follows from the fact that $u$ is of type $(0,3)$ and the Bianchi identity that
$$\Delta_{\overline{\partial}_A} u = \overline{\partial}_A \overline{\partial}_A^* u = \overline{\partial}_A F_A^{0,2} =0.$$
Taking the inner product with $u$ and integrating by parts we get that $\Vert \overline{\partial}_A^* u \Vert^2_{L^2} =0$. Thus $F_A^{0,2}=0$ and so $(A,u)$ solves 
\begin{align}\label{eq:J_Integrable}
F_A^{0,2} & =   0, \nonumber \\ 
\overline{\partial}_A^* u & =  0, \\  \nonumber
\Lambda F_A & =   \ast [ u \wedge \overline{ u} ],
\end{align}
In particular, this proves the following. 

\begin{proposition}
	If $X$ is compact and $J$ integrable, then any DT-instanton $(A, u)$ solves \ref{eq:J_Integrable}. In particular, $A$ induces an holomorphic structure on any associated complex vector bundle.
\end{proposition}

In fact, notice that as $u$ is of type $(0,3)$ we have that $\overline{\partial}_A^* u = - \ast \partial_A \ast u = -i \ast \partial_A u$, so that $\overline{u} \in H^0(X, K_X \otimes \mathfrak{g}^{\mathbb{C}}_P)$, where $\mathfrak{g}^{\mathbb{C}}_P =\mathfrak{g}_P \otimes_{\mathbb{R}} \mathbb{C}$.

\subsection{On K\"ahler manifolds}

For K\"ahler manifolds we can prove that if the scalar curvature is positive then $u=0$ and the equations \ref{eq:J_Integrable} coincide with the HYM equations. We state this as follows.

\begin{proposition}
	Let $X$ be compact and $\omega$ be a K\"ahler metric with nonnegative scalar curvature $s \geq 0$ on $X$. Then, any DT-instanton $(A,u)$ on $(X, \omega)$ has $A$ being HYM and $u =0$.
\end{proposition}
\begin{proof}
	Using the Weitzenb\"ock formula 
	$$2 \delbar_A \delbar_A^* u = \nabla_A^* \nabla_A u + \frac{s}{4} u +[i \Lambda F_A , u] $$
	and the last equation $\Lambda F_A =   \ast [ u \wedge \overline{ u} ]$ we compute
	\begin{eqnarray*}
		\frac{1}{2} \Delta |u|^2 & = & \Re (\langle \nabla_A^* \nabla_A u , \overline{ u} \rangle ) - |\nabla_A u|^2 \\
		& = & - \frac{s}{4} |u|^2 - \Re (\langle [ i\Lambda F_A , u] , \overline{ u} \rangle ) - |\nabla_A u|^2   \\
		& = & - \frac{s}{4} |u|^2 - |[u \wedge \overline{ u}]|^2 - |\nabla_A u|^2  .
	\end{eqnarray*}
	Thus, if the scalar curvature $s$ of the K\"ahler metric is nonnegative we must have $u=0$ and the equations above reduce to the HYM ones.
\end{proof}

\begin{remark}
Suppose $u \in L^4$, which is a natural assumption from the variational point of view and also follows from the assumption that $u \in L^2$ and Moser iteration as shown in \cite{Tanaka2013}. Then, integrating the computation above shows that $\nabla_A u \in L^2$. 
\end{remark}
%
%The previous computation also has very nice consequences on other K\"ahler manifolds as a consequenece of Moser's iteration, see \cite{Tanaka2013}.
%
%\begin{proposition}
%	
%\end{proposition}
%\begin{proof}
%	It follows from the computation in the proof of the previous proposition that \begin{eqnarray*}
%		\frac{1}{2} \Delta |u|^2 & = &  - \frac{s}{4} |u|^2 - |[u \wedge \overline{ u}]|^2 - |\nabla_A u|^2  \\
%		& \leq & c |u|^2,
%	\end{eqnarray*}
%	where $4c = \min \lbrace - \inf_{x \in X} s(x) , 1 \rbrace$. Hence, Moser's iteration yields an estimate of the kind 
%	$$\sup_{X} u \leq C \| u \|_{L^2(X)},$$
%	for some $C>0$.
%\end{proof}
%
%
%

\subsection{For compatible $\SU(3)$-structures}

One other interesting case when the equations simplify is when the almost Hermitian structure admits a compatible $\SU(3)$-structure. Of course, this is a very restrictive condition. Indeed, an almost Hermitian structure admits a compatible $\SU(3)$-structure if and only if the canonical bundle $K_X:=\Lambda^{3,0}_{\mathbb{C}}X$ is (topologically) trivial. When this is the case we shall say that a compatible $\SU(3)$-structure is \emph{pseudo-holomorphcially} trivial if there is a $\Omega$ so that $\overline{\partial} \Omega =0$. Under this further restriction we can rewrite the equations as follows.

\begin{proposition}\label{prop:New_Mon_Eq}
Suppose there is compatible $\SU(3)$-structure $(J,\omega, \Omega)$ satisfying $\overline{\partial} \Omega=0$. Let $\Phi_1, \Phi_2 \in \Omega^0(X, \mathfrak{g})$ and write $u= \frac{i}{4}(\Phi_1 + i \Phi_2) \overline{\Omega}$. Then, equations \ref{eq:New_Mon1}--\ref{eq:New_Mon2} may be written as
\begin{eqnarray}\label{eq:mon1}
\ast d_A \Phi_1 & = & F_A \wedge \Omega_1 - d_A \Phi_2 \wedge \frac{\omega^2}{2}, \\ \label{eq:mon2}
F_A \wedge \frac{\omega^2}{2} & = & [\Phi_1 , \Phi_2] \frac{\omega^3}{3!}.
\end{eqnarray}
Equivalently, the first equation \ref{eq:mon1} may be written as
\begin{equation}\label{eq:mon3}
\ast d_A \Phi_2 = F_A \wedge \Omega_2 + d_A \Phi_1 \wedge \frac{\omega^2}{2}.
\end{equation}
\end{proposition}
\begin{proof}
We shall use the notation $\Phi=\Phi_1+i \Phi_2$. We start by inserting $u$ as in the statement into the first equation \ref{eq:New_Mon1}
\begin{eqnarray}\nonumber
F^{0,2} & = &   \overline{\partial}_A^* \left( \frac{i}{4} \Phi \overline{\Omega} \right) =-\frac{i}{4}\ast  \partial_A \left( \Phi \ast \overline{\Omega} \right) = \frac{1}{4}\ast   \left( \partial_A \Phi \wedge \overline{\Omega} \right),
\end{eqnarray}
where we used $\ast \overline{\Omega}=  i \overline{\Omega}$ and the hypothesis that $\overline{\partial} \Omega =0$. Next, wedge this equation with $\Omega$ 
$$F \wedge \Omega = \frac{1}{4}\ast   \left( \partial_A \Phi \wedge \overline{\Omega} \right) \wedge \Omega = 2 \ast \partial_A \Phi,$$ 
where we used the fact that the projection $\Omega^1 \rightarrow \Omega^{1,0}$ can be written as $8a^{1,0}= - \ast ( \ast (a \wedge \overline{\Omega}) \wedge \Omega)$, for $a \in \Omega^1$. Then, separate this equation into types and use the fact that 
$$\ast ( d_A \Phi_i \wedge \omega^2 /2 )= - J d_A \Phi_i$$ 
to obtain equations \ref{eq:mon1} and \ref{eq:mon3}. Finally, equation \ref{eq:mon2}, which follows from inserting $u= \frac{i}{4} \Phi \overline{\Omega}$, using $\Omega \wedge \overline{\Omega}=-8i \dvol$ and $\Lambda F_A = \ast ( F_A \wedge \omega^2/2 )$.
\end{proof}

In several cases the DT-instanton equations \ref{eq:New_Mon1}--\ref{eq:New_Mon2} reduce to the pHYM equations, which further motivates studying the DT-instantons as an extension of the pHYM connections for the general almost Hermitian structure. In this direction, we start by proving that if there is a compatible half-flat $\SU(3)$-structure, then the DT instanton equations \ref{eq:mon1}--\ref{eq:mon2} reduce to a simpler equation with only one Higgs field.

\begin{proposition}\label{prop:Half_Flat_Vanishing}
Let $X$ be compact and $(J,\omega, \Omega)$ be a half-flat $\SU(3)$-structure. Then, any irreducible DT-instanton $(A,u)$ for a simple Lie group $G$, satisfies $\Phi_1=0$ and
\begin{eqnarray}
I d_A \Phi_2 & = & - \ast (F_A \wedge \Omega_1 ) \\ \label{eq:mon1_Half_Flat}
F_A \wedge \frac{\omega^2}{2} & = & 0, \label{eq:mon2_Half_Flat}
\end{eqnarray}
or, rewriting the first equation, $d_A \Phi_2 = - \ast( F_A \wedge \Omega_2 ) $.
\end{proposition}
\begin{proof}
First, notice that if the $\SU(3)$ structure is half flat then $d\Omega_1=0$ and so $d\Omega=id\Omega_2$ is real. However, by type decomposition 
$d \Omega = \overline{\partial} \Omega + N(\Omega) $
is of type $(3,1)+(2,2)$, where $N$ is the Nijenhuis tensor. Hence, $\overline{\partial} \Omega =0$ and we can write the DT instanton equations as in \ref{eq:mon1}--\ref{eq:mon2}. Now, equip the bundle $\mathfrak{g}_{P}$ with an $\Ad$-invariant metric compatible with $A$, and compute
$$\Delta \frac{\vert \Phi_1 \vert^2}{2} = \langle \Phi_1 , \Delta_A \Phi_1 \rangle - \vert \nabla_A \Phi_1 \vert^2.$$
Using equation \ref{eq:mon1} together with the Bianchi identity and $d \Omega_1=0=d \omega^2$, we have
\begin{eqnarray}\nonumber
\Delta_A \Phi_1 & = & - \ast d_A \left( F_A \wedge \Omega_1 - d_A \Phi_2 \wedge \frac{\omega^2}{2} \right)\\ \nonumber
 & = & \ast [F_A , \Phi_2] \wedge \frac{\omega^2}{2} \\ \nonumber
 & = & [[\Phi_1, \Phi_2], \Phi_2],
\end{eqnarray} 
where in the last equality we used equation \ref{eq:mon2}. Hence $\langle \Delta_A \Phi_1 , \Phi_1 \rangle = - \vert [\Phi_1 , \Phi_2] \vert^2$ and so 
$$\Delta \frac{\vert \Phi_1 \vert^2}{2} = - \vert [\Phi_1, \Phi_2] \vert^2 - \vert \nabla_A \Phi_1 \vert^2,$$
is subharmonic. As $X$ is compact, $\vert \Phi_1 \vert$ is constant and so the previous equation yields $[\Phi_1, \Phi_2]=0$ and $\nabla_A \Phi_1 =0$. Hence, given that the connection is irreducible, fixing a trivialization of $E$ at a point $p \in X$, $\Phi_1(p)$ is a central element, However, as $G$ is semisimple, we must have $\Phi_1=0$. The remaining equations follows simply from inserting $\Phi_1=0$ into equations \ref{eq:mon1}--\ref{eq:mon2}.
\end{proof}

For a compact Calabi-Yau manifold the DT-instanton equation further reduces to the HYM equations. Indeed, if the $\SU(3)$-structure is Calabi-Yau it also is half-flat and the DT instanton equations can be written as \ref{eq:mon1_Half_Flat}--\ref{eq:mon2_Half_Flat}, with $\Phi_1=0$. Moreover, from the Bianchi identity and $d\Omega_2=0$ it follows that
$$\Delta \Phi_2 = \ast d_A \ast d_A \Phi_2 = \ast d_A (F_A \wedge \Omega_2) =0.$$
Thus $\Delta \vert \Phi_2 \vert^2 = -2 \vert \nabla_A \Phi_2 \vert^2 \leq 0$, and again, if $X$ is compact, $G$ semisimple and $A$ is irreducible must have $\Phi_2 =0$. Hence, equations yields \ref{eq:mon1_Half_Flat}--\ref{eq:mon2_Half_Flat}, so that $A$ is actually HYM. In the case of noncompact Calabi-Yau manifolds, irreducible DT-instantons with $\overline{\partial}^* u \neq 0$ for semisimple Lie groups do exist, and are expected to be related to special Lagrangian submanifolds, see \cite{Oliveira2014} for more on this. On nearly K\"ahler manifolds, a similar vanishing result holds true, as we now state.

\begin{proposition}\label{prop:NKMon}
Let $X$ be compact\footnote{Any nearly K\"ahler manifold is Einstein with positive scalar curvature. Hence, if it is complete, must actually be compact.} and $(J,g)$ be compatible with a nearly K\"ahler structure. Then, any DT-instanton $(A,u)$ satisfies
\begin{eqnarray}\label{eq:HYM1}
 F_A^{0,2}  & = & 0 \\ \label{eq:HYM2}
F_A \wedge \omega^2 & = & 0,
\end{eqnarray}
and $\nabla_A \Phi_1 = \nabla_A \Phi_2 = [\Phi_1, \Phi_2] =0$, with $\Phi_1, \Phi_2$ as in proposition \ref{prop:New_Mon_Eq}. In particular, $A$ is a pHYM connection.
\end{proposition}
\begin{proof}
As any nearly K\"ahler structure is half-flat, the same proof as before yields that proposition \ref{prop:New_Mon_Eq} applies and $(A,u)$ can be written as in \ref{eq:mon1}--\ref{eq:mon2} (In fact, also proposition \ref{prop:Half_Flat_Vanishing} applies and we could start applying the result therein). Here, we proceed in three steps from equations \ref{eq:mon1}--\ref{eq:mon3}.\\

\noindent \textbf{Step 1:} We prove that
$$\Delta_A \Phi_1 = [[\Phi_1 , \Phi_2] , \Phi_2 ] \ , \ \Delta_A \Phi_2 =  4 [\Phi_1 , \Phi_2]  - [ [\Phi_1 , \Phi_2] , \Phi_1 ].$$
This follows from equation \ref{eq:mon3}, the Bianchi identity and the equations for a nearly K\"ahler $\SU(3)$-structure. Indeed,
\begin{eqnarray}\nonumber
\Delta_A \Phi_2 & = & - \ast d_A \ast d_A \Phi_2 = - \ast d_A \left(  F_A \wedge \Omega_2 + d_A \Phi_1 \wedge \frac{\omega^2}{2} \right) \\ \nonumber
& = & - \ast \left( -2 F_A \wedge \omega^2 + [ F_A , \Phi_1 ] \wedge \frac{\omega^2}{2} \right),
\end{eqnarray}
and inserting here equation \ref{eq:mon2} gives
\begin{eqnarray}\label{eq:IntDelta}
\Delta_A \Phi_2 & = &   4 [\Phi_1 , \Phi_2]  - [ [\Phi_1 , \Phi_2] , \Phi_1 ] .
\end{eqnarray}
The case of $\Phi_1$ follows from a similar, but easier, computation, close to the Calabi-Yau case in lemma $3.1.14$ of \cite{Oliveira2014}.\\

\noindent \textbf{Step $2$}: We prove that
\begin{eqnarray}\nonumber
\Delta \frac{\vert \Phi_1 \vert^2}{2} = - \vert [\Phi_1 , \Phi_2] \vert^2 - \vert \nabla_A \Phi_1 \vert^2 \ , \ \Delta \frac{\vert \Phi_2 \vert^2}{2} = - \vert [\Phi_1 , \Phi_2] \vert^2 - \vert \nabla_A \Phi_2 \vert^2 
\end{eqnarray}
As before, we shall only prove the case of $\Phi_2$. This follows from inserting \ref{eq:IntDelta} into the following computation 
\begin{eqnarray}\nonumber
\Delta \frac{\vert \Phi_2 \vert^2}{2} & = & \langle \Phi_2 , \Delta_A \Phi_2 \rangle  - \vert \nabla_A \Phi_2 \vert^2 \\  \nonumber
& = &  4 \langle \Phi_2 , [\Phi_1 , \Phi_2] \rangle - \langle [ [\Phi_1 , \Phi_2] , \Phi_1 ] , \Phi_2 \rangle  - \vert \nabla_A \Phi_2 \vert^2 \\ \nonumber
& = &  - \vert  [\Phi_1 , \Phi_2] \vert^2  - \vert \nabla_A \Phi_2 \vert^2,
\end{eqnarray}
where we used the $\Ad$-invariance of the inner product on $\mathfrak{g}_P$.\\

\noindent \textbf{Step 3}: We finish the proof by noticing that, from step $2$, both $\vert \Phi_1 \vert^2$, $\vert \Phi_2 \vert^2$ are subharmonic and since $X$ is compact, they must actually be constant. Hence, their Laplacians vanish and once again step $2$ gives that $\nabla_A \Phi_1=\nabla_A \Phi_2 = [\Phi_1, \Phi_2]=0$ and the DT-instanton equations reduce to the pHYM ones.
\end{proof}

\begin{remark}
In the K\"ahler case the Dolbeaut splitting passes to cohomology and Hodge theory proves that there is a unique (up to gauge) HYM connection on any complex line bundle $L$ whose first Chern class $c_1(L)$ if of type $(1,1)$. This has degree $0$ in the case when $c_1(L)$ is primitive. In fact, the curvature of this HYM connection is the harmonic representative of $c_1(L)$. Similarly, it follows from a result of Lorenzo Foscolo \cite{Foscolo2016} that on a nearly K\"ahler manifold the harmonic representative of any degree-$2$ cohomology class is primitive of type $(1,1)$. Hence, by a repeated use of the Poincar\'e lemma, one proves that for any $\alpha \in H^2(X, \mathbb{Z})$ on a nearly K\"ahler manifold there is a complex line bundle $L$ with $c_1(L)=\alpha$ and which can be equipped with a pHYM connection unique up to gauge.
\end{remark}

%%%%%%%%%%%%%%%%%%%%%%%%%%%%%%%%%%%%%%%%%%%%%%%%%%%%

\section{Invariant almost-Hermitian structures on $\mathbb{F}_{2}$}

The manifold of full flags in $\mathbb{C}^3$ is the homogeneous manifold $\SU(3)/ T^2$, where $T^2$ is the subgroup of diagonal matrices, a maximal torus in $\SU(3).$ As in \cite{Bryant2005}, we use the Maurer-Cartan form on $\SU(3)$ given by 
$$g^{-1} dg =  \left[ \begin {array}{ccc} i\beta_{{1}}&\theta_{{3}}+i\eta_{{3}}&-
\theta_{{2}}+i\eta_{{2}}\\ \noalign{\medskip}-\theta_{{3}}+i\eta_{{3}}
&i\beta_{{2}}&\theta_{{1}}+i\eta_{{1}}\\ \noalign{\medskip}\theta_{{2}
}+i\eta_{{2}}&-\theta_{{1}}+i\eta_{{1}}&i\beta_{{3}}\end {array}
\right] .
$$
Consider the set of simple roots $S=\lbrace{ r_i \rbrace}_{i=1}^3$ such that the $r_i \in S \subset (\mathfrak{t}^2)^*$ are given by
$$r_1=i\beta_1 +2i\beta_2, \ \ r_2=-2i \beta_1-i\beta_2, \ \ r_3=i\beta_1 -i\beta_2,$$
and set $\mathfrak{m}_i= (\mathfrak{sl}_{r_i}(3, \mathbb{C}) \oplus \mathfrak{sl}_{-r_i}(3, \mathbb{C}) ) \cap \mathfrak{su}(3) $, to be the real component of the root spaces. These are respectively
$$\mathfrak{m}_1^* = \langle \eta_1 , \theta_1 \rangle , \ \mathfrak{m}_2^* = \langle \eta_2 , \theta_2 \rangle , \ \mathfrak{m}_3^* = \langle \eta_3 , \theta_3 \rangle . $$
Then, we fix the complement $\mathfrak{m}$ to the isotropy $\mathfrak{t}^2 \subset \mathfrak{su}(3)$ so that $\mathfrak{m}^* =\mathfrak{m}^*_1 \oplus \mathfrak{m}^*_2 \oplus \mathfrak{m}^*_3$.

We would like to consider $\SU(3)$-invariant almost complex structures $J$. Evaluating any such $J$ at the identity coset and extending it by left invariance one obtains an $(\Ad,T^2)$-invariant map $J: \mathfrak{m} \to \mathfrak{m}$ with $J^2=-\id_{\mathfrak{m}}$. From Schur's lemma any such map must preserve the root spaces $\mathfrak{m}_i$ and we shall now describe them by fixing a trivialization for the pullback of $\Lambda^{1,0}_{\mathbb{C}}$ to $\SU(3)$.

For each $A_1,A_2,A_3 \in \mathbb{R}^+$ and $\epsilon_1, \epsilon_2, \epsilon_3 \in \lbrace \pm 1 \rbrace$, we define the complex valued $1$-forms
\begin{equation}\label{eq:Complex_Structures_Flag}
\begin{aligned}\nonumber
\alpha_1 & =  A_1(\eta_1 + i \epsilon_1 \theta_1) \\  
\alpha_2 & =  A_2(\eta_2 + i \epsilon_2 \theta_2) \\  
\alpha_3 & =  A_3(\eta_3 + i \epsilon_3 \theta_3) .
\end{aligned}
\end{equation}
We then consider the $\SU(3)$-invariant almost complex structures $J$ so that the $\alpha_i$ span the pullback of $\Lambda^{1,0}_{\mathbb{C}}$ to $\SU(3)$. Using the Maurer-Cartan equations we compute the Nijenhuis tensor of $J$. This is diagonal in the basis $\lbrace \alpha_i \rbrace_{i=1}^3$, $\lbrace \overline{\alpha}_j \wedge \overline{\alpha}_k \rbrace_{j<k}$ of $\Lambda^{1,0}_{\mathbb{C}}$ and $\Lambda^{0,2}_{\mathbb{C}}$ respectively, being given by $N=\diag (n_{11},n_{22},n_{33})$ with
\begin{eqnarray}\label{eq:Nijenhuis_Flag}
\begin{aligned}
n_{11} & =  \frac{1}{4} \,{\frac {A_{{1}} \left( \epsilon_1 \epsilon_2 +\epsilon_1 \epsilon_3 + \epsilon_2 \epsilon_3 +1 \right) }{A_{{2}}\epsilon_{{2}}A_{{3}}\epsilon_{{3}}}} \\ 
n_{22} & =  \frac{1}{4} \,{\frac {A_{{2}} \left(   \epsilon_1 \epsilon_2 +\epsilon_1 \epsilon_3 + \epsilon_2 \epsilon_3 +1 \right) }{A_{{3}}\epsilon_{{3}}A_{{1}}\epsilon_{{1}}}} \\ 
n_{33} & =  \frac{1}{4} \,{\frac {A_{{3}} \left(  \epsilon_1 \epsilon_2 +\epsilon_1 \epsilon_3 + \epsilon_2 \epsilon_3 +1 \right) }{A_{{1}}\epsilon_{{1}}A_{{2}}\epsilon_{{2}}}}.
\end{aligned}
\end{eqnarray}

\begin{remark}[Weyl group and symmetries]
	The Weyl group of $\SU(3)$ acts on $\mathbb{F}$ and so on the almost complex structures determined by $(\epsilon_1,\epsilon_2,\epsilon_3)$. Having in mind that $r_3=-(r_1+r_2)$ we easily find that the Weyl group is generated by the reflections $p_2,p_2,p_3$ such that $p_1(r_1,r_2,r_3)=(-r_1,-r_3,-r_2)$ and similarly for $p_2$ and $p_3$. In particular, $\sigma = p_2 \circ p_1$ is an element of order $3$ which cyclically permutes the roots $r_1,r_2,r_3$.
\end{remark}

In particular, from equations \ref{eq:Nijenhuis_Flag}, the almost complex structure $J$ is integrable if and only if
\begin{equation}\label{eq:Integrable_Flag}
\epsilon_1 \epsilon_2 +\epsilon_1 \epsilon_3 + \epsilon_2 \epsilon_3 +1 =0.
\end{equation}
In particular, up to action of the Weyl group there are only two invariant almost complex structures $J^i$ and $J^{nk}$ with $(\epsilon_1,\epsilon_2,\epsilon_3)$ respectively determined by $(1,1,-1)$ and $(1,1,1)$. In particular, inserting these into equation \ref{eq:Integrable_Flag} we find that while $J^i$ is integrable $J^{ni}$ is not. Indeed, we shall see shortly that $J^{ni}$ is compatible with a nearly K\"ahler structure on $\mathbb{F}_2$. Now we consider the almost Hermitian structures determined by setting
\begin{eqnarray}\label{eq:SU(3)structure_Flag1}
\omega & = & \frac{i}{2} (\alpha_1 \wedge \overline{\alpha}_1 + \alpha_2 \wedge \overline{\alpha}_2 + \alpha_3 \wedge \overline{\alpha}_3 ) ,
\end{eqnarray}
which always satisfies the equation
\begin{equation}\label{eq:omega^2_Is_Closed}
d \omega^2 =0.
\end{equation} 
In fact, a tedious but otherwise straightforward computation shows that $d \omega = \Re(\gamma)$, where
\begin{eqnarray}\label{eq:Gamma}
\begin{aligned}
\gamma & =  \frac{A_1^2 \epsilon_1 + A_2^2 \epsilon_2 + A_3^2 \epsilon_3}{4A_1 A_2 A_3} \frac{1}{\epsilon_1 \epsilon_2 \epsilon_3} \lbrace (\epsilon_1 \epsilon_2 \epsilon_3 + \epsilon_1 + \epsilon_2 + \epsilon_3) \alpha_{123} \\ 
&   + (\epsilon_1 \epsilon_2 \epsilon_3 + \epsilon_1 - \epsilon_2 - \epsilon_3) \overline{\alpha_1} \wedge \alpha_2 \wedge \alpha_3 + (\epsilon_1 \epsilon_2 \epsilon_3 - \epsilon_1 + \epsilon_2 - \epsilon_3) \alpha_1 \wedge \overline{\alpha_2} \wedge \alpha_3  \\ 
&   + (\epsilon_1 \epsilon_2 \epsilon_3 - \epsilon_1 - \epsilon_2 + \epsilon_3) \alpha_1 \wedge \alpha_2 \wedge \overline{\alpha_3}  \rbrace.
\end{aligned}
\end{eqnarray}
As the terms $\epsilon_i \epsilon_j \epsilon_k - \epsilon_i - \epsilon_j + \epsilon_k$ and $\epsilon_1 \epsilon_2 \epsilon_3 + \epsilon_1 + \epsilon_2 + \epsilon_3$ cannot all vanish at the same time, it follows that $\omega$ is symplectic if and only if
\begin{equation}\label{eq:Sympletic_Flag}
A_1^2 \epsilon_1 + A_2^2 \epsilon_2 + A_3^2 \epsilon_3 =0.
\end{equation}
In particular, the $\epsilon_i$ cannot all have the same sign and we conclude that $J^{ni}$ cannot tame any symplectic form. On the other hand, $J^i$ is compatible with a real $2$-parameter family of symplectic structures determined by solving \ref{eq:Sympletic_Flag}. In fact, these are all nonequivalent as for any such there are three different holomorphic spheres of different areas. These correspond to three orbits of the $\U(2)$ subgroup tangent to each of the root spaces $\mathfrak{m}_i$ at the origin, which have area $A_i^2$ and so must be different for any two such $(A_1,A_2,A_3)$ up to changes of sign. 
	
\begin{example}[K\"ahler-Einstein structure]\label{ex:Kahler-Einstein}
	The K\"ahler structure determined by $J^i$ and $\omega$ with $A_3^2=2A_1^2=2A_2^2$ can be seen to be Einstein. This is the standard homogeneous K\"ahler-Einstein Fano structure on $\mathbb{F}_2$.
\end{example}

The $3$-form on $\SU(3)$ given by
\begin{eqnarray} \label{eq:SU(3)structure_Flag2}
\Omega & = & \Omega_1 + i \Omega_2= \alpha_1 \wedge \alpha_2 \wedge \alpha_3,
\end{eqnarray}
is semi-basic and of type $(3,0)$. For $J^{ni}$, i.e. $\epsilon_1=\epsilon_2=\epsilon_3=1$, it is actually basic and so descends to a nowhere vanishing $(3,0)$-form on $\mathbb{F}_2$. In particular, $c_1(T\mathbb{F}_2, J^{ni})=0$ and $(\omega, \Omega)$ determines a $\SU(3)$-structure on $\mathbb{F}_2$ compatible with $(\omega,J^{ni})$. 

\begin{example}[Nearly K\"ahler almost complex structure]\label{ex:nearly_Kahler_str_Flag}
	For $\epsilon_1=\epsilon_2=\epsilon_3= 1$, i.e. $J^{ni}$ we have that the $3$-form $\gamma$ in \ref{eq:Gamma}, so that $d\omega = \Re (\gamma)$, is given by
	\begin{eqnarray}\nonumber
	\gamma = \frac{A_1^2  + A_2^2  + A_3^2 }{A_1 A_2 A_3} \alpha_{123} =  \frac{A_1^2  + A_2^2  + A_3^2 }{2A_1 A_2 A_3} \Omega.
	\end{eqnarray}
	Thus, any $J^{ni}$-compatible invariant almost Hermitian structure satisfies
	\begin{equation}
	d \omega = \frac{A_1^2  + A_2^2  + A_3^2 }{2A_1 A_2 A_3} \ \alpha_{123} =  \frac{A_1^2  + A_2^2  + A_3^2 }{A_1 A_2 A_3} \ \Omega_1 ,
	\end{equation}
	As a consequence, for any $(A_1,A_2,A_3)$ we have $d \Omega_1=0$, which together with equation \ref{eq:omega^2_Is_Closed} shows that any $\SU(3)$-structure in this real $3$-dimensional family is half flat. In particular, when $A_1=A_2=A_3=1$ 
	$$d\omega= 3\Omega_1 , \ d\Omega_2 = -2\omega^2 ,$$
	and so the $\SU(3)$-structure is the homogeneous nearly K\"ahler structure. 
\end{example}

\section{pHYM connections on $\U(1)$-bundles over $\mathbb{F}_2$}

In this section we prove proposition \ref{prop:U(1)_pHYM_Flag} and corollary \ref{cor:U(1)_pHYM_Flag} below which describe $\SU(3)$-invariant pseudo-holomorphic and pHYM connections on complex line bundles over $\mathbb{F}_2$.

We shall start by describing homogeneous circle bundles, invariant connections and Higgs fields. Topologically any circle bundle is determined by a class in $\mathrm{H}^2(\mathbb{F}_2, \mathbb{Z})$, the first Chern class of the complex line bundle associated with the standard representation. The Serre spectral sequence for the fibration $T^2 \rightarrow \SU(3) \rightarrow \mathbb{F}_2$ then gives an isomorphism
\begin{equation}\label{eq:Serre}
\mathrm{H}^2(\mathbb{F}_2, \mathbb{Z}) \cong \mathrm{H}^1(T^2 , \mathbb{Z}) = \Hom (\Lambda , \mathbb{Z}) ,
\end{equation}
where $\Lambda = \ker( \exp: \mathfrak{t}^2 \rightarrow T^2)$ is the weight lattice. Then, given such an integral weight $\beta \in \Hom (\Lambda , \mathbb{Z}) \subset \Hom(\mathfrak{t}^2 , \mathbb{R})$ we view it as an $1$-form in $\mathfrak{t}^2$ which we extend to $\mathfrak{su}(3)$, using the orthogonal splitting induced by the Killing form. Finally, using left invariance we further extend $\beta$ to a $1$-form in $\SU(3)$. In fact, the $\mathfrak{u}(1)=i\mathbb{R}$ valued $1$-form $i \beta$ in $\SU(3)$ is a connection on the complex line bundle 
$$L_{\beta}=\SU(3) \times_{(T^2, e^{i\beta} )} \mathbb{C}.$$
Its first Chern class can be computed from its curvature, using the identifications in equation \ref{eq:Serre}, this is once again 
$$c_1(L_{\beta}) = \frac{i}{2 \pi} [d ( i \beta )] = - \frac{1}{2 \pi} [d \beta ] \in  \mathrm{H}^2(\mathbb{F}_2, \mathbb{Z}),$$
which corresponds to $ \beta \in \Hom (\Lambda , \mathbb{Z}) $ under the isomorphism \ref{eq:Serre}.

Recall that for all almost Hermitian structures under consideration $d \omega^2=0$ and so $[\omega^2] \in H^4(\mathbb{F}_2,\mathbb{Z})$ yields a well defined cohomology class. Thus it makes sense to define 
$$\deg(L_{\beta}):= 2 \pi \langle c_1(L_{\beta}) \cup [\omega^2] , [\mathbb{F}_2] \rangle ,$$
and the slope 
$$\mu (L_{\beta}):= \frac{\deg(L_{\beta})}{\mathrm{Vol}(\mathbb{F}_2)}.$$
Writing $\beta = k \beta_1 + l \beta_2$ we have from the Maurer-Cartan equation
\begin{equation}\label{eq:d_beta}
d\beta = \frac{il}{\epsilon_1 A_1^2}  \alpha_1 \wedge \overline{\alpha}_1 - \frac{ik}{\epsilon_2 A_2^2}  \alpha_2 \wedge \overline{\alpha}_2 + \frac{i(k-l)}{\epsilon_3 A_3^2}  \alpha_3 \wedge \overline{\alpha}_3.
\end{equation}
Thus, we immediately compute that
\begin{equation}\label{eq:Degree}
\mu(L_{\beta}) = \frac{2}{3} \left( - \frac{l}{\epsilon_1 A_1^2} + \frac{k}{\epsilon_2 A_2^2 } - \frac{k-l}{\epsilon_3 A_3^2} \right) .
\end{equation}

For circle bundles, the adjoint bundle is a trivial real line bundle and so a Higgs field is simply a real valued function. Moreover, the induced connection on the adjoint bundle is also trivial. If we assume that the Higgs field is invariant under the action of $\SU(3)$, then $\Phi$ must be constant, and thus any invariant DT-instanton must actually be a pHYM connection together with a constant Higgs field. As a consequence, in considering invariant DT instantons on circle bundles, there is no loss of generality in simply considering the pHYM equations.

\begin{proposition}\label{prop:U(1)_pHYM_Flag}
	Given a complex line bundle $L_{\beta}$, the connection $i\beta$ described above is always pseudo-holomorphic. Furthermore, it is pHYM with degree $0$ if and only if 
	$$\mu(L_{\beta})=0.$$	
	If $i\beta= k i\beta_1 + l i\beta_2$ for $k,l \in \mathbb{Z}$, this condition can be equivalently written as
	$$k\epsilon_{{1}} A_1^2 ({A_{{2}}}^{2}\epsilon_{{2}}-{A_{{3}}}^{2}\epsilon_{{3}} )=l\epsilon_{{2}}{A_{{2}}}^{2} ( {A_{{1}}}^{2}
	\epsilon_{{1}}-{A_{{3}}}^{2}\epsilon_{{3}}).$$
\end{proposition}
\begin{proof}
	For $\beta= k \beta_1 + l \beta_2$ the group homomorphism $e^{i \beta} : T^2 \rightarrow \U(1)$ is $e^{i \beta}(e^{i s_1}, e^{i s_2})=e^{i(ks_1+ls_2)}$. Recall the bundle is constructed via $L_{\beta}=\SU(3) \times_{(\SU(3),  e^{i \beta})} \U(1)$. Wang's theorem \cite{Wang1958} shows that any invariant connection on $L_{\beta}$ differs from $i \beta$ by the addition of the left invariant extension of a morphism of $T^2$-representations $(\mathfrak{m}, \Ad) \rightarrow (\mathfrak{u}(1), \Ad \circ  e^{i \beta})$. As the later one is trivial and $\mathfrak{m}$ has no trivial $T^2$ components, Schur's lemma yields that $i \beta$ is the only invariant connection. Its curvature is $F_{\beta}=id\beta$ as computed in equation \ref{eq:d_beta}. We see that $F_{\beta}$ is of type $(1,1)$ and so the connections $\beta$ are all pseudo-holomorphic. For $\beta$ to be pHYM of degree $0$ we must have $F_{\beta} \wedge \omega^2=0$ which implies that
	$$k\epsilon_{{1}} A_1^2 ({A_{{2}}}^{2}\epsilon_{{2}}-{A_{{3}}}^{2}\epsilon_{{3}} )=l\epsilon_{{2}}{A_{{2}}}^{2} ( {A_{{1}}}^{2}
	\epsilon_{{1}}-{A_{{3}}}^{2}\epsilon_{{3}}).
	$$
\end{proof}

In particular, the bundle $L_{\beta}$ need not be trivial for it to admit a pHYM connection of degree $0$. In fact, the equation $k\epsilon_{{1}} A_1^2 ({A_{{2}}}^{2}\epsilon_{{2}}-{A_{{3}}}^{2}\epsilon_{{3}} )=l\epsilon_{{2}}{A_{{2}}}^{2} ( {A_{{1}}}^{2}
\epsilon_{{1}}-{A_{{3}}}^{2}\epsilon_{{3}})
$ vanishes for all $l,k$ if and only if $A_1^2 \epsilon_1 = A_2^2 \epsilon_2 = A_3^2 \epsilon_3$. For instance, the $\epsilon_i$ must all have the same sign and so the almost complex structure must be the non-integrable one $J^{ni}$ in which case we have $A_1^2  = A_2^2  = A_3^2 $ and the almost Hermitian is the nearly K\"ahler one up to scaling. This proves the following result.

\begin{corollary}\label{cor:U(1)_pHYM_Flag}
	The $\SU(3)$-structure \ref{eq:SU(3)structure_Flag1}--\ref{eq:SU(3)structure_Flag2} admits an invariant pHYM connection of degree $0$ on all complex line bundles if and only if 
	\begin{equation}\label{eq:Existence_of_U(1)_pHYM}
	A_1^2 \epsilon_1 = A_2^2 \epsilon_2 = A_3^2 \epsilon_3.
	\end{equation}
	In particular the almost Hermitian structure is the nearly K\"ahler one of example \ref{ex:nearly_Kahler_str_Flag} up to scaling.
\end{corollary}

\begin{remark}
In the K\"ahler case, follows easily from Hodge theory that a HYM  connection on a complex line bundle always exists and is unique. In that case, it has degree $0$ if and only if the bundle itself has degree $0$. In the nearly K\"ahler case, any complex line bundle has degree $0$ as $\omega^2$ is exact. Also in that case one can use Hodge theory to prove that a pHYM connection always exists and is unique (up to gauge), see \cite{Foscolo2016}.
\end{remark}

\section{DT-instantons on $\SO(3)$-bundles over $\mathbb{F}_2$}\label{sect:DTSO3}

In this section we classify $\SU(3)$-invariant pHYM connections and DT-instantons with gauge group $\SO(3)$.

The isomorphism classes of homogeneous $\SO(3)$-bundles are parametrized by group homomorphisms $\lambda:T^2 \rightarrow \SO(3)$. Thinking of $\SO(3)$ as $\SU(2)/\mathbb{Z}_2$ any such homomorphism is of the form
$$\lambda_{\beta} = \diag( e^{\frac{i}{2}\beta} , e^{-\frac{i}{2}\beta} ) \in \SU(2)/\mathbb{Z}_2 ,$$
and the corresponding homogeneous $\SO(3)$-bundle is
$$P_\beta = \SU(3) \times_{(T^2, \lambda_\beta)} \SO(3).$$

\begin{proposition}\label{prop:Basic_Monopoles_SU(3)}
	Let $\beta$ be an integral weight and $(A, \Phi_1 , \Phi_2)$ be an irreducible $\SU(3)$-invariant triple on $P_{\beta}$ whose pullback to $\SU(3)$ satisfies equations \ref{eq:mon1}--\ref{eq:mon2}. Then, $\beta$ is a root of $\SU(3)$ and $(A, \Phi_1, \Phi_2)$ is given by:
	\begin{itemize}
		\item If $\beta= r_1 = \beta_1 + 2\beta_2$, in which case $\epsilon_1 \epsilon_2 - \epsilon_2 \epsilon_3 + \epsilon_1 \epsilon_3 = 1$ and $\epsilon_1 \mu(L_{r_1})<0$, i.e. $A_1^2 A_3^2 \frac{\epsilon_1}{\epsilon_2} + A_1^2 A_2^2 \frac{\epsilon_1}{\epsilon_3} < 2 A_2^2 A_3^2 $. In this case 
		\begin{align}\nonumber
		& \Phi_1  =  0, \\ \nonumber
		& A  =  r_1 \otimes \frac{T_1}{2} \pm \sqrt{ 1 - \frac{A_1^2}{2 A_2^2} \frac{\epsilon_1}{\epsilon_2} - \frac{A_1^2}{2 A_3^2} \frac{\epsilon_1}{\epsilon_3} } \ (\eta_1 \otimes T_2 - \theta_1 \otimes T_3),\\ \nonumber
		& \Phi_2  = - \frac{A_1}{A_2 A_3} \frac{\epsilon_2 \epsilon_3 + 1}{\epsilon_2 \epsilon_3} T_1.
		\end{align}
		
		\item If $\beta= r_2 =-2\beta_1 - \beta_2$, in which case $\epsilon_2 \epsilon_3 - \epsilon_3 \epsilon_1 + \epsilon_2 \epsilon_1 = 1$ and $\epsilon_2 \mu(L_{r_2})<0$, i.e. $A_2^2 A_1^2 \frac{\epsilon_2}{ \epsilon_3} + A_2^2 A_3^2 \frac{\epsilon_2}{ \epsilon_1} < 2 A_3^2 A_1^2 $. In this case 
		\begin{align}\nonumber
		& \Phi_1  =  0, \\ \nonumber
		& A  =  r_2 \otimes \frac{T_1}{2} \pm \sqrt{ 1 - \frac{A_2^2}{2 A_3^2} \frac{\epsilon_2}{\epsilon_3} - \frac{A_2^2}{2 A_1^2} \frac{\epsilon_2}{\epsilon_1} } \ (\eta_2 \otimes T_2 - \theta_2 \otimes T_3),\\ \nonumber
		& \Phi_2  = - \frac{A_2}{A_3 A_1} \frac{\epsilon_3 \epsilon_1 + 1}{\epsilon_3 \epsilon_1} T_1.
		\end{align}
		
		\item If $\beta= r_3 = \beta_1 - \beta_2$, in which case $\epsilon_3 \epsilon_1 - \epsilon_1 \epsilon_2 + \epsilon_3 \epsilon_2 = 1$ and $\epsilon_3 \mu(L_{r_3})<0$, i.e. $A_3^2 A_2^2 \frac{\epsilon_3}{ \epsilon_1} + A_3^2 A_1^2 \frac{\epsilon_3}{ \epsilon_2} < 2 A_1^2 A_2^2$. In this case 
		\begin{align}\nonumber
		& \Phi_1  =  0, \\ \nonumber
		& A  =  r_3 \otimes \frac{T_1}{2} \pm \sqrt{ 1 - \frac{A_3^2}{2 A_1^2} \frac{\epsilon_3}{\epsilon_1} - \frac{A_3^2}{2 A_2^2} \frac{\epsilon_3}{\epsilon_2} } \ (\eta_3 \otimes T_2 - \theta_3 \otimes T_3),\\ \nonumber
		& \Phi_2  = - \frac{A_3}{A_1 A_2} \frac{\epsilon_1 \epsilon_2 + 1}{\epsilon_1 \epsilon_2} T_1.
		\end{align}
	\end{itemize}
	Moreover, when equality, rather than strict inequality, holds in any of the above cases the corresponding DT instanton becomes reducible and $A$ is one of the pHYM connections from proposition \ref{prop:U(1)_pHYM_Flag}.
\end{proposition}
\begin{proof}
	For each $\beta = k \beta_1 + l \beta_2$, with $(k,l) \in \mathbb{Z}^2$, the group homomorphism $\lambda_{\beta}$ is given by 
	$$\lambda_{\beta}(e^{is_1}, e^{is_2}) = \diag(e^{i(ks_1+ls_2)} , e^{-i(ks_1+ls_2)}) \in \SU(2)/\mathbb{Z}_2 .$$
	The canonical invariant connection $A_{\beta}^c= d \lambda_{\beta} \oplus 0$ on $P_{k,l}$ is determined by the $1$-form in $\SU(3)$ with values in $\mathfrak{so}(3)$ given by
	$$A_{\beta}^c = \beta \otimes \frac{T_1}{2} =  \left( k \beta_1 + l \beta_2 \right) \otimes \frac{T_1}{2}.$$
	By Wang's theorem, other invariant connections on $P_{\beta}$ are determined by morphisms of $T^2$-representations $\Lambda_{\beta}: (\mathfrak{m}, \Ad) \rightarrow (\mathfrak{so}(3), \Ad \circ \lambda_{\beta})$. The left and right hand sides respectively decompose into irreducible components as $\mathbb{C}_{r_1} \oplus \mathbb{C}_{r_2} \oplus \mathbb{C}_{r_3}$ and $\mathbb{R} \oplus \mathbb{C}_{\beta}$. Hence, other invariant connections exist only in the case when $\beta = r_i$ for some $i =1,2,3$. In all other cases the canonical invariant connection is the unique one, and the Ambrose-Singer theorem shows that any such connection is reducible, so that we would be back in the case analysed in proposition \ref{prop:U(1)_pHYM_Flag}. It is then enough to restrict ourselves to the three cases above. We shall start with the case $\beta= k \beta_1 + l \beta_2 = r_1$ so $k=1$ and $l=2$. In this case the most general invariant connection can be written as
	\begin{equation}\label{eq:Invariant_Connection_F2}
	A= A_{r_1}^c + a (\eta_1 \otimes T_2 - \theta_1 \otimes T_3),
	\end{equation}
	for $a \in \mathbb{R}$. Its curvature can be computed as $F_A=dA + \frac{1}{2} [A \wedge A ]$ and we can write it as $F=F_1 \otimes T_1+ F_2 \otimes T_2 + F_3 \otimes T_3$ with
	\begin{eqnarray}\nonumber
	F_1 & = & \frac{i}{\epsilon_1 A_1^2} (1- a^2) \alpha_1 \wedge \overline{\alpha}_1 - \frac{i}{2 \epsilon_2 A_2^2} \alpha_2 \wedge \overline{\alpha}_2 - \frac{i }{2 \epsilon_3 A_3^2} \alpha_3 \wedge \overline{\alpha}_3    \\ \nonumber
	F_2 & = &  - \frac{ a(\epsilon_2 + \epsilon_3)}{ 2\epsilon_2 \epsilon_3 A_2 A_3} \Im( \alpha_2 \wedge \alpha_3) + \frac{a( \epsilon_2 -\epsilon_3 )}{2 \epsilon_2 \epsilon_3 A_2 A_3} \Im( \alpha_2 \wedge \overline{\alpha}_3 ) \\ \nonumber
	F_3 & = & \frac{a(1+ \epsilon_2 \epsilon_3)}{2 \epsilon_2 \epsilon_3 A_2 A_3} \Re ( \alpha_2 \wedge \alpha_3 ) + \frac{a( \epsilon_2 \epsilon_3 -1)}{2 \epsilon_2 \epsilon_3 A_2 A_3} \Re ( \alpha_2 \wedge \overline{\alpha}_3 ) .
	\end{eqnarray}
	Again, it follows from the Ambrose-Singer theorem that for such a connection to be irreducible we need $a \neq 0$. We turn now to invariant Higgs fields $\Phi \in \Omega^0(\mathbb{F}_2 , \Ad(\mathfrak{g}_{P_{r_1}}))$. We view these as functions in $\SU(3)$ with values in $\mathfrak{g} \cong \mathfrak{su}(2)$, equivariant with respect to the action of $T^2 \subset \SU(3)$ on $\SU(3)$ by multiplication on the right and on $\mathfrak{su}(2)$ via $\Ad \circ \lambda_{r_1}$. For $\Phi$ to be left-invariant it must be constant, and so valued in the trivial component of $\mathfrak{su}(2) \cong \mathbb{R} \oplus \mathbb{C}_{r_1}$. Thus, we may write our two invariant Higgs fields as $\Phi_i = - \phi_i T_1$, for $i=1,2$, where $\phi_1$, $\phi_2$ are real numbers. Their covariant derivative with respect to the connection $A$, as in equation \ref{eq:Invariant_Connection_F2}, can be computed to be 
	\begin{eqnarray}\nonumber
	d_A \Phi_i & = & d\Phi_i + [A, \Phi_i] \\ \nonumber 
	& = & \frac{2a\phi_i}{\epsilon_1 A_1} \Im(\alpha_1) \otimes T_2 + \frac{2a\phi_i}{A_1} \Re(\alpha_1) \otimes T_3.
	\end{eqnarray}

	Now recall the DT instanton equations \ref{eq:mon1}--\ref{eq:mon2}, which we rewrite here for convenience
	\begin{eqnarray}\label{eq:mon1_Recall}
	\ast d_A \Phi_1 & = & F_A \wedge \Omega_1 - d_A \Phi_2 \wedge \frac{\omega^2}{2} \\ \label{eq:mon2_Recall}
	F_A \wedge \frac{\omega^2}{2} & = & [\Phi_1 , \Phi_2] \frac{\omega^3}{3!}.
	\end{eqnarray}
	We start with the second equation \ref{eq:mon2_Recall}. In our case we have that $[\Phi_1, \Phi_2]=0$ and so the equation turns into $F_A \wedge \omega^2 =0$. The component of $F_1 \wedge \omega^2 =0$ is given by
	$$ 2A_2^2 A_3^2 \epsilon_2 \epsilon_3 a^2 + A_1^2 A_2^2 \epsilon_1 \epsilon_2 +A_1^2 A_3^2 \epsilon_1 \epsilon_3-2 A_2^2 A_3^2 \epsilon_2 \epsilon_3 =0$$
	while the equations $F_2 \wedge \omega^2 =0$ and $F_3 \wedge \omega^2 =0$ hold automatically. It then follows that we must have
	$$\frac{a^2}{\epsilon_1 A_1^2} = \frac{1}{\epsilon_1 A_1^2} - \frac{1}{2 \epsilon_2 A_2^2}  - \frac{1}{2 \epsilon_3 A_3^2} = - \frac{3}{4} \mu(L_{r_1}) ,$$
	and so a solution exists if and only if $ -\epsilon_1 \mu(L_{r_1}) >0 $, in which case
	\begin{equation}\label{eq:a_proof_Monopole_main_thm}
	a = \pm \sqrt{-\frac{3 \epsilon_1 A_1^2}{4} \mu(L_{r_1})}.
	\end{equation} 
	Moreover, by the previous comment, the resulting connection is irreducible if and only if strict inequality holds.
	
	We now turn to the first of the DT instaton equations, i.e. equation \ref{eq:mon1_Recall}. To compute the right hand side we must start by computing the Hodge star operator. Given that in the frame $\lbrace \Re(\alpha_i), \Im(\alpha_i) \rbrace_{i=1}^3$ and that $\omega^3/3!= \Re(\alpha_1) \wedge \Im(\alpha_1) \wedge \ldots \wedge \Im(\alpha_3)$, we compute that $\ast \Re(\alpha_1)= \Im(\alpha_1) \wedge \ldots \wedge \Re(\alpha_3) \wedge \Im(\alpha_3)$ and $\ast \Im(\alpha_1) = - \Re(\alpha_1) \wedge \Re(\alpha_2) \wedge \ldots \wedge \Im(\alpha_3)$. Thus, the left hand side of equation \ref{eq:mon1_Recall} is
	\begin{eqnarray}\nonumber
	\ast d_A \Phi_1 & = & - \frac{2a\phi_1}{\epsilon_1 A_1} \Re(\alpha_1) \wedge \Re(\alpha_2) \wedge \Im(\alpha_2) \wedge \Re(\alpha_3) \wedge \Im(\alpha_3) \otimes T_2 \\ \nonumber
	&  & + \frac{2a\phi_1}{ A_1} \Im(\alpha_1) \wedge \Re(\alpha_2) \wedge \Im(\alpha_2) \wedge \Re(\alpha_3) \wedge \Im(\alpha_3) \otimes T_3 \\ \nonumber
	& = & - \frac{a\phi_1}{2\epsilon_1 A_1} \Re(\alpha_1) \wedge \alpha_2 \wedge \overline{\alpha}_2 \wedge \alpha_3 \wedge \overline{\alpha}_3 \otimes T_2 \\ \nonumber
	&  & + \frac{a\phi_1}{2 A_1} \Im(\alpha_1) \wedge \alpha_2 \wedge \overline{\alpha}_2 \wedge \alpha_3 \wedge \overline{\alpha}_3 \otimes T_3 .
	\end{eqnarray}
	As for the right hand side, i.e. $F_A \wedge \Omega_1 - d_A \Phi_2 \wedge \omega^2/2$, the component along $T_1$ is simply $F_1 \wedge \Omega_1$ which vanishes identically, while the other components are
	\begin{align}\nonumber
	& F_A \wedge \Omega_1  -  d_A \Phi_2 \wedge \frac{\omega^2}{2}  =  \\ \nonumber
	& = \frac{a( - 2 A_{{2}}A_{{3}}  \epsilon_{{2}}\epsilon_{{3}}\phi_{{2}}+A_{{1}} \epsilon_1 (\epsilon_2 + \epsilon_3) )}{4 \epsilon_1 \epsilon_2 \epsilon_3 A_1 A_2 A_3} \Im(\alpha_1) \wedge \alpha_2 \wedge \overline{\alpha}_2 \wedge \alpha_3 \wedge \overline{\alpha}_3 \otimes T_2 \\ \nonumber
	& + \frac{a( -2 A_{{2}}A_{{3}} \epsilon_1 \epsilon_{{2}}\epsilon_{{3}}\phi_{{2}}+A_{{1}} \epsilon_{{1}} (\epsilon_{{2}}\epsilon_{{3}}+1) )}{4 \epsilon_1 \epsilon_2 \epsilon_3 A_1 A_2 A_3} \Re(\alpha_1) \wedge \alpha_2 \wedge \overline{\alpha}_2 \wedge \alpha_3 \wedge \overline{\alpha}_3 \otimes T_3.
	\end{align}
	So the remaining equations turn into
	\begin{align*}\nonumber
	& a \phi_1 = 0 , \\ \nonumber
	& a( - 2 A_{{2}}A_{{3}}  \epsilon_{{2}}\epsilon_{{3}}\phi_{{2}}+A_{{1}} \epsilon_1 (\epsilon_2 + \epsilon_3) ) = 0 , \\ \nonumber
	& a( -2 A_{{2}}A_{{3}} \epsilon_1 \epsilon_{{2}}\epsilon_{{3}}\phi_{{2}}+A_{{1}} \epsilon_{{1}} (\epsilon_{{2}}\epsilon_{{3}}+1) ) = 0.
	\end{align*}
	Then, either:
	\begin{itemize}
		\item $a=0$ and $\phi_1 \in \mathbb{R}$ is free to choose in which case the connection is reducible to $\U(1) \subset \SO(3)$ and we are back in the case analyzed in the previous section (see proposition \ref{prop:U(1)_pHYM_Flag} and corollary \ref{cor:U(1)_pHYM_Flag}).
		
		\item or $\phi_1=0$, in which case we can then impose the remaining equations. When the connection $A$ is irreducible, i.e. when $a \neq 0$ is given by equation \ref{eq:a_proof_Monopole_main_thm}, the two remaining equations above are compatible if and only if
		$$\epsilon_1 \epsilon_2 - \epsilon_2 \epsilon_3 + \epsilon_1 \epsilon_3 = 1.$$
		In that case we can solve for $\phi_2$ yielding
		$$\phi_2 = \frac{A_1}{A_2 A_3} \frac{\epsilon_1 (\epsilon_2 +\epsilon_3)}{\epsilon_2 \epsilon_3} = \frac{A_1}{A_2 A_3} \frac{\epsilon_2 \epsilon_3 + 1}{\epsilon_2 \epsilon_3}.$$
		The cases when $\beta$ is either $r_2$ or $r_3$ are similar and in fact can be obtained from this one by applying $\sigma$, the element of order $3$ in the Weyl group of $\SU(3)$.
	\end{itemize} 
\end{proof}

\begin{remark}\label{rk:GaugeEquiv}
	Notice that in each case the two connections differing from a choice of sign in $a = \pm \sqrt{-3 \epsilon_i A_i^2 \mu(L_{r_i}) / 4}$ are actually gauge equivalent. The gauge transformation exchanging these is given by $\exp(\frac{\pi}{2} T_1)$.
\end{remark}

\subsection{pHYM connections}

We shall start by using the results of proposition \ref{prop:Basic_Monopoles_SU(3)} to analyse the existence of $\SU(3)$-invariant pHYM connections with structure group $\SO(3)$. We shall use these same results to analyse the existence of DT-Instantons and compare them with those we now obtain for HYM connections. Indeed, as we will now see, the existence of irreducible pHYM connections implies the almost complex structure is actually integrable and so the pHYM connections are actually HYM.

\begin{proposition}
	Let $\beta$ be an integral weight and $A$ an irreducible pHYM connection on $P_{\beta}$ for $\mathbb{F}_2$ equipped with an invariant almost Hermitian structure. Then, the almost complex structure is in fact integrable and either 
	\begin{itemize}
		\item $\beta= r_1$, in which case $\epsilon_2 = \pm 1$, $\epsilon_3= \mp 1$ and $\epsilon_1 \mu(L_{r_1})<0$, i.e. $ \mp \epsilon_1 A_1^2 A_2^2 \pm \epsilon_1 A_1^2 A_3^2  < 2 A_2^2 A_3^2 $, and
		\begin{equation}\nonumber
		A  =  r_1 \otimes \frac{T_1}{2} \pm \sqrt{ 1 \mp \epsilon_1 A_1^2 \left( \frac{1}{2 A_2^2}  - \frac{1}{2 A_3^2} \right)  } \ (\eta_1 \otimes T_2 - \theta_1 \otimes T_3) .
		\end{equation}
		
		\item $\beta= r_2$, in which case $\epsilon_3=\pm 1$ and $\epsilon_1 = \mp 1$ and $\epsilon_2 \mu(L_{r_2})<0$, i.e.  $\pm  \epsilon_2 A_2^2 A_1^2 \mp  \epsilon_2 A_2^2 A_3^2  < 2 A_3^2 A_1^2 $, and
		\begin{align}\nonumber
		& A  =  r_2 \otimes \frac{T_1}{2} \pm \sqrt{ 1 \mp \epsilon_2 A_2^2 \left( \frac{1}{2 A_3^2} - \frac{1}{2 A_1^2} \right)  } \ (\eta_2 \otimes T_2 - \theta_2 \otimes T_3).
		\end{align}
		
		\item If $\beta= r_3$, in which case $\epsilon_1 = \pm 1$, $\epsilon_2 = \mp 1$ and $\epsilon_3 \mu(L_{r_3})<0$, i.e.  $\pm \epsilon_3 A_3^2 A_2^2  \mp \epsilon_3 A_3^2 A_1^2  < 2 A_1^2 A_2^2 $. In this case 
		\begin{align}\nonumber
		& A  =  r_3 \otimes \frac{T_1}{2} \pm \sqrt{ 1 \mp \epsilon_3 A_3^2 \left( \frac{1}{2 A_1^2} - \frac{1}{2 A_2^2} \right)  } \ (\eta_3 \otimes T_2 - \theta_3 \otimes T_3).
		\end{align}
	\end{itemize}
	Moreover, when equality, rather than strict inequality, holds in any of the above cases the connection $A$ becomes reducible.
\end{proposition}
\begin{proof}
	The computations for an invariant pHYM connection are contained, as a subcase, in the DI-instanton ones. They correspond to the DT-instantons for which $\Phi_1=0=\Phi_2$. As can be seen from the statement of proposition \ref{prop:Basic_Monopoles_SU(3)}, this happens if and only if $\epsilon_i \epsilon_j =-1$ for some $i,j \in \lbrace 1,2,3 \rbrace$.  In the proof we shall only deal with the case of $\epsilon_3 = - \epsilon_2^{-1}$ as the other ones follow similar lines. Then, the condition that $\Phi_2$ vanishes in proposition \ref{prop:Basic_Monopoles_SU(3)} yields that $\beta=r_1$, while the condition that $\epsilon_1 \epsilon_2 - \epsilon_2 \epsilon_3 + \epsilon_1 \epsilon_3 = 1$ turns into $\epsilon_1 (\epsilon_2 - \epsilon_2^{-1})=0$ and so $\epsilon_2 = \pm 1$ with $\epsilon_3 = \mp 1$ respectively. Inserting this into the inequality involving the metric structure, i.e. the $A_i$, we must have $ \pm \epsilon_1 A_1^2 A_3^2  \mp \epsilon_1 A_1^2 A_2^2  <  2 A_2^2 A_3^2 $, which is the inequality in the statement.\\
	All the almost complex structures to which this result applies, i.e. those satisfying $\epsilon_i= \pm 1$, $\epsilon_j = \mp 1$ and $\epsilon_k \in \mathbb{R}$, are in fact integrable. Indeed, it is easy to check that for any such, the Nijenhuis tensor computed in equation \ref{eq:Nijenhuis_Flag}, vanishes identically.
\end{proof}

Using the action of the Weyl group we may fix the invariant integrable complex structure to be $J^i$, i.e. $(\epsilon_1,\epsilon_2,\epsilon_3)=(1,1,-1)$, and so Theorem \ref{thm:HYM_Integrable_Flag} implies that for irreducible invariant HYM to exist on $P_{\beta}$, $\beta$ must either be $r_1$ or $r_2$ and substituting into the previous theorem we obtain the following.

\begin{theorem}\label{thm:HYM_Integrable_Flag}
	For any invariant Hermitian structure $(g,J)$ on $\mathbb{F}_2$ there are, up to gauge, at most two invariant irreducible HYM connections with gauge group $\SO(3)$. Using the action of the Weyl group so that $J=J^i$, these are the following:
	\begin{itemize}
		\item The connection
		\begin{equation}\nonumber
		A  =  r_1 \otimes \frac{T_1}{2} \pm \sqrt{ 1 -  A_1^2 \left( \frac{1}{2 A_2^2}  - \frac{1}{2 A_3^2} \right)  } \ (\eta_1 \otimes T_2 - \theta_1 \otimes T_3) ,
		\end{equation} 
		on $P_{r_1}$ which exists in case $\mu(L_{r_1})<0$, i.e. $ A_1^2 ( A_3^2 -A_2^2)  < 2 A_2^2 A_3^2 $.
		
		\item The connection
		\begin{align}\nonumber
		& A  =  r_2 \otimes \frac{T_1}{2} \pm \sqrt{ 1 + A_2^2 \left( \frac{1}{2 A_3^2} - \frac{1}{2 A_1^2} \right)  } \ (\eta_2 \otimes T_2 - \theta_2 \otimes T_3),
		\end{align}
		on $P_{r_2}$, which exists in case $\mu(L_{r_2})<0$, i.e. $A_2^2 (A_3^2 -A_1^2)  < 2 A_3^2 A_1^2 $.
	\end{itemize}
	Moreover, when equality, rather than strict inequality, holds in any of the above cases the connection $A$ becomes reducible.
\end{theorem}

Suppose that one can find an invariant Hermitian structure which admits no invariant irreducible HYM connection with gauge group $\SO(3)$, then we would have that both $A_1^2 ( A_3^2 -A_2^2) \geq 2 A_2^2 A_3^2$ and $A_2^2 (A_3^2 -A_1^2) \geq 2 A_3^2 A_1^2$. Summing these two equations we obtain
$$ - 2A_1^2 A_2^2  \geq  A_2^2 A_3^2 +  A_3^2 A_1^2, $$
which is obviously impossible. Thus we conclude the following

\begin{corollary}\label{cor:pHYM_Kahler_1}
	All invariant Hermitian structures on $\mathbb{F}_2$ admit invariant, irreducible HYM connections with gauge group $\SO(3)$. 
\end{corollary}

\begin{remark}
	This result also follows as an application of the universal Hitchin-Kobayashi correspondence \cite{LT}.
\end{remark}

\begin{example}
	Consider the K\"ahler-Einstein structure from example \ref{ex:Kahler-Einstein}. Up to scaling, this is given $A_3^2=2A_1^2=2A_2^2=2A^2$ and $(\epsilon_1 ,\epsilon_2,\epsilon_3)=(1,1,- 1)$ and using these together with \ref{thm:HYM_Integrable_Flag} we immediately see that irreducible HYM connections on $P_{r_1}$ and $P_{r_2}$ exist.
\end{example}

Finally we shall now prove one last consequence of Theorem \ref{thm:HYM_Integrable_Flag}.

\begin{corollary}\label{cor:pHYM_Kahler_2}
	There is a family of invariant K\"ahler structures $\lbrace \omega_s \rbrace_{s \in I \subset \mathbb{R}}$ with the following property. There is $s_0 \in I$, such that: for $s<s_0$ the bundle has two irreducible, invariant HYM connections; these converge to the same reducible and obstructed HYM connection, as $s \rightarrow s_0$; and for $s>s_0$ there are no irreducible, invariant HYM connections.\footnote{The two irreducible HYM connections existing for $s<s_0$ are actually gauge equivalent, see remark \ref{rk:GaugeEquiv}. However, the gauge transformation exchanging them fixed the reducible HYM connection existing at $s=s_0$.}
\end{corollary}
\begin{proof}
	We shall explicitly construct a family $\lbrace \omega_s \rbrace_{s \in I \subset \mathbb{R}}$ explicitly. Let $\epsilon <1/10$ be positive and $I=(1-\epsilon , 1+\epsilon)$, then we set 
	$$A_1 = 1, \ A_2=\frac{1}{\sqrt{2+\sqrt{3}}}, \ A_3=s,$$
	and $\epsilon_1 = s^2 - \frac{1}{2+ \sqrt{3}}$, which is positive (and less than $1$) for $s \in I$. Then, by the proof of the first part we have that for $s< s_0=1$ there are two irreducible, invariant HYM connections on $P_{r_1}$; while for $s>s_0$ there are no invariant HYM connections with gauge group $\SO(3)$. The fact that the connections become obstructed an reducible as $s \rightarrow s_0=1$ follows from a straightforward computation.
\end{proof}

\subsection{DT-instantons}

Recall that in general the $3$-form $\Omega$ is only semibasic. However, it is basic for $\epsilon_1=\epsilon_2=\epsilon_3=  1$, which corresponds to the almost complex structure is $J^{ni}$. We shall now analyse the consequences of proposition \ref{prop:Basic_Monopoles_SU(3)} for the existence theory of DT-instantons for invariant almost Hermitian structures.

\begin{theorem}\label{thm:Monopoles_Flag_JnK}
	Equip $\mathbb{F}_2$ with an invariant almost Hermitian structure compatible with $J^{nk}$. Let $\beta$ be an integral weight and $(A, u)$ an irreducible DT-instanton on $P_{\beta} \rightarrow \mathbb{F}_2$. Then, $\beta$ is a root of $\SU(3)$ and the DT-Instanton can be written, as in proposition \ref{prop:New_Mon_Eq}, in terms of $(A, \Phi_1 ,\Phi_2)$ given by:
	\begin{itemize}
		\item If $\beta= r_1$, and $\mu(L_{r_1})<0$, i.e. $A_1^2 A_3^2 + A_1^2 A_2^2  < 2 A_2^2 A_3^2 $. In this case 
		\begin{align}\nonumber
		& \Phi_1  =  0, \\ \nonumber
		& A  =  r_1 \otimes \frac{T_1}{2} \pm \sqrt{ 1 - \frac{A_1^2}{2 A_2^2}  - \frac{A_1^2}{2 A_3^2}  } \ (\eta_1 \otimes T_2 - \theta_1 \otimes T_3),\\ \nonumber
		& \Phi_2  = - \frac{2A_1}{A_2 A_3} T_1.
		\end{align}
		
		\item If $\beta= r_2 $, and $\mu(L_{r_2})<0$, i.e. $A_2^2 A_1^2 + A_2^2 A_3^2 < 2 A_3^2 A_1^2 $. In this case 
		\begin{align}\nonumber
		& \Phi_1  =  0, \\ \nonumber
		& A  =  r_2 \otimes \frac{T_1}{2} \pm \sqrt{ 1 - \frac{A_2^2}{2 A_3^2}  - \frac{A_2^2}{2 A_1^2}  } \ (\eta_2 \otimes T_2 - \theta_2 \otimes T_3),\\ \nonumber
		& \Phi_2  = - \frac{2A_2}{A_3 A_1} T_1.
		\end{align}
		
		\item If $\beta= r_3$, and $\mu(L_{r_3})<0$, i.e. $A_3^2 A_2^2 + A_3^2 A_1^2 < 2 A_1^2 A_2^2$. In this case 
		\begin{align}\nonumber
		& \Phi_1  =  0, \\ \nonumber
		& A  =  r_3 \otimes \frac{T_1}{2} \pm \sqrt{ 1 - \frac{A_3^2}{2 A_1^2}  - \frac{A_3^2}{2 A_2^2} } \ (\eta_3 \otimes T_2 - \theta_3 \otimes T_3),\\ \nonumber
		& \Phi_2  = - \frac{2A_3}{A_1 A_2} T_1.
		\end{align}
	\end{itemize}
	Moreover, when equality, rather than strict inequality, holds in any of the above cases the corresponding DT-instanton becomes reducible and $A$ is one of the pHYM connections described in proposition \ref{prop:U(1)_pHYM_Flag}.
\end{theorem}
\begin{proof}
	First we note that, as remarked before, $d \omega^2=0$ for any $\SU(3)$-invariant almost Hermitian structure. Moreover, when the almost complex structure is $J^{ni}$, i.e. when $\epsilon_1=\epsilon_2=\epsilon_3= \pm 1$ we can compute that $d \Omega$ is of type $(2,2)$ and so we can apply proposition \ref{prop:New_Mon_Eq} and write the DT-instanton equations as in \ref{eq:mon1}--\ref{eq:mon2}. The computations of these equations have been performed in the proof of proposition \ref{prop:Basic_Monopoles_SU(3)} and the current theorem is simply a particular case of that result.
\end{proof}

\begin{remark}
	The fact that $\Phi_1=0$ in all the DT-instantons above is also a consequence of proposition \ref{prop:Half_Flat_Vanishing}. Indeed, a computation shows that for of the $\SU(3)$-structures above compatible with $J^{ni}$ we further have $d\Omega_1=0$ and so proposition \ref{prop:Half_Flat_Vanishing} applies.
\end{remark}

\begin{corollary}\label{cor:Monopoles_Jnk_Flag}
	For any invariant almost Hermitian structure on $\mathbb{F}_2$ compatible with $J^{ni}$ there exists at least one $\SO(3)$-bundle equipped with a DT-instanton. These DT-instantons are reducible, in which case $A$ is a pHYM connection, if and only if the almost Hermitian structure is nearly K\"ahler (up to scaling).
\end{corollary}
\begin{proof}
	Suppose there is no such DT-instanton, not even the reducible ones when equality is achieved in the strict inequalities of theorem \ref{thm:Monopoles_Flag_JnK}. Then, we must have that
	\begin{align*}
	& A_1^2 A_3^2 + A_1^2 A_2^2  > 2 A_2^2 A_3^2 , \\
	& A_2^2 A_1^2 + A_2^2 A_3^2  > 2 A_3^2 A_1^2 , \\
	& A_3^2 A_2^2 + A_3^2 A_1^2  > 2 A_1^2 A_2^2 ,
	\end{align*}
	and adding these up we arrive at a contradiction. Hence, at least one of the inequalities above is violated and either: $A_i^2 A_j^2 + A_i^2 A_k^2 < 2 A_j^2 A_k^2$ for one or two permutations $(i,j,k)$ of $(1,2,3)$, or equality is achieved in all. In the first case, theorem \ref{thm:Monopoles_Flag_JnK} yields the existence of a bundle supporting at least one DT-instanton. In the second case we must have that all the equalities hold. This implies that $A_1^2=A_2^2=A_3^2$ and so, up to scaling, the metric is equivalent to the nearly K\"ahler one.
\end{proof}

\begin{example}\label{ex:Mon_1}
	Consider the family of almost Hermitian structures $(g,J^{ni})$ with $g$ determined by $A_1=A_2=1$ and $A_3=x$. For $x=\pm 1$ this is the almost Hermitian structure compatible with the nearly K\"ahler one. In figure \ref{fig:1}, we have depicted $x$ on the horizontal axis, and in the vertical axis the value of $a$ (as in the proof of \ref{prop:Basic_Monopoles_SU(3)}). This component of the connection has the property that it is nonzero if and only if the DT instanton is irreducible, with each pair of colors corresponding to the values of $a$ for the different weights $\beta$. Analysing this figure we see that for any value of $x$, other than $x=\pm 1$ the DT-instantons are irreducible. When $x= \pm 1$, i.e. when the metric is the nearly K\"ahler one, these become obstructed and reducible to one of the Abelian pHYM connections in proposition \ref{prop:U(1)_pHYM_Flag}.
	\begin{figure}\label{fig:1}
		\centering
		\includegraphics[scale=0.45]{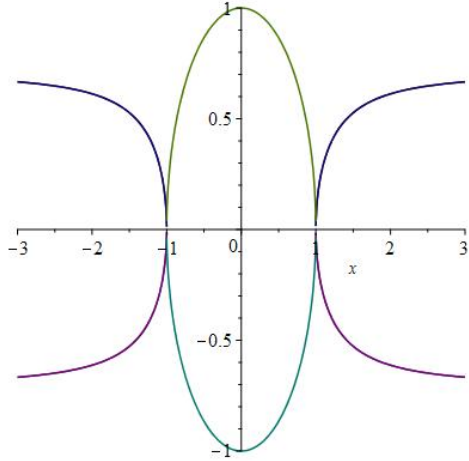}
		\caption{\label{fig} With $J^{ni}$ and $A_1=A_2=1$, $A_3=x$.}
	\end{figure}
\end{example}

\begin{example}\label{ex:Mon_2}
	Now consider the family of almost Hermitian structures $(g,J^{nk})$ with $g$ determined by $A_1=x$, $A_2=10x^3$ and $A_3=1$. In figure \ref{fig:2} we produce a similar plot as in the previous case. In this case we see that there are irreducible DT-instantons for any value of the parameter $x$.
	\begin{figure}\label{fig:2}
		\centering
		\includegraphics[scale=0.45]{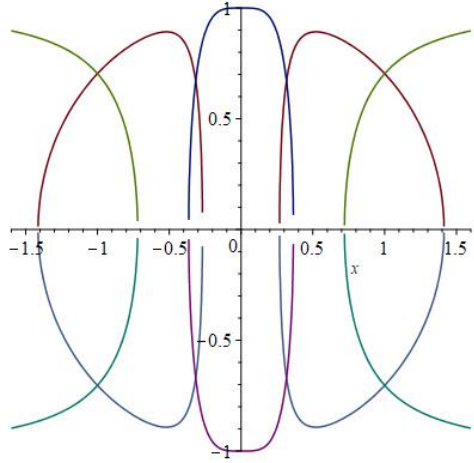}
		\caption{\label{fig2} With $J^{ni}$ and $A_1=x$, $A_2=10x^3$, $A_3=1$.}
	\end{figure}
\end{example}

\appendix

\section{The topology of the bundles $P_{\beta}$}

Recall that the bundles $P_{\beta}$ are constructed via $P_{\beta}= \SU(3) \times_{(T^2, \lambda_{\beta})} \SO(3)$, where
$$\lambda_{\beta} = \diag( e^{\frac{i}{2}\beta} , e^{-\frac{i}{2}\beta} ) \in \SU(2)/\mathbb{Z}_2 .$$
Let $V_{\beta}=P_{\beta} \times_{SO(3)} \mathbb{R}^3$ be the vector bundle associated with respect to the standard representation $SO(3)$-representation, and consider the $\U(2)$ bundle 
$$E_{\beta}=\SU(3) \times_{(T^2 , \tilde{\lambda}_{\beta} )} \mathbb{C}^2 ,$$ where
$$\tilde{\lambda}_{\beta} = \diag( e^{i\beta} , 0 ) \in  \U(2), $$
and the ismorphism. This has the property that the $\U(2)$-adjoint bundle of $E_{\beta}$ splits as $\mathfrak{u}_{E_{\beta}} \cong \underline{\mathbb{R}} \oplus V_{\beta}$ and
$$w_2(V_{\beta})= c_1(E_{\beta})  \mod 2, \ \ \ p_1(V_{\beta}) = c_1(E_{\beta})^2 - 4 c_2(E_{\beta}) . $$
We shall now compute the Chern classes of the bundles $E_{\beta}$ using Chern-Weyl theory. For this we must equip $E_{\beta}$ with a connection which we choose to be the standard invariant connection given by
$$A_{\beta}= \beta \otimes \diag(i,0).$$
This has curvature $F_{\beta} = d\beta \otimes \diag(i,0)$ and so
$$c_1(E_{\beta})=-\frac{1}{2\pi}[d \beta] , \ \ c_2(E_{\beta})= \frac{1}{4 \pi^2} [d \beta] \cup [d \beta].$$
Furthermore, a computation using the Maurer-Cartan equations shows that
$$d\beta_1^2 + d\beta_2^2 + d \beta_1 \wedge d\beta_2 = d \Im ((\eta_1+i\theta_1) \wedge (\eta_1+i\theta_1) \wedge (\eta_1+i\theta_1)),$$
and so in $H^4(\mathbb{F}_2, \mathbb{Z})$ we have
$$[d\beta_1] \cup [d \beta_2 ]= - [d\beta_1] \cup [d \beta_1 ] - [d\beta_2] \cup [d \beta_2 ].$$
So, writing $\beta = k \beta_1 + l \beta_2$ we compute that
$$w_2(V_{\beta})=-\frac{1}{2\pi}\left( k[d \beta_1] + l [d \beta_2] \right) \mod 2 ,$$
while
\begin{eqnarray}\nonumber
p_1(V_{\beta}) & = & \frac{1}{4 \pi^2} [d \beta] \cup [d \beta]  \\ \nonumber
& = & \frac{1}{4 \pi^2} \left( k^2 [d \beta_1] \cup [d \beta_1 ]  + 2kl  [d \beta_1] \cup [d \beta_2] + l^2  [d \beta_2] \cup [d \beta_2] \right) \\ \nonumber
& = & \frac{1}{4 \pi^2} \left(  k(k-2l) [d \beta_1] \cup [d \beta_1 ] + l(l-2k)  [d \beta_2] \cup [d \beta_2] \right) .
\end{eqnarray}

In particular, when $\beta$ is one of the roots $r_1$, $r_2$, $r_3$ we respectively obtain

\begin{align*}
& w_1(P_{r_1}) = -\frac{1}{2\pi} [d \beta_1]  \mod 2 , \ \ \ p_1(V_{r_1})= -\frac{3}{4 \pi^2} [d \beta_1] \cup [d \beta_1 ]  , \\ 
& w_1(P_{r_2}) =  \frac{1}{2\pi}[d \beta_2] \mod 2 , \ \ \ p_1(V_{r_2})= -\frac{3}{4 \pi^2}   [d \beta_2] \cup [d \beta_2] , \\
& w_1(P_{r_3}) = -\frac{1}{2\pi}\left( [d \beta_1] - [d \beta_2] \right) \mod 2 , \ \ \ p_1(V_{r_3})= \frac{1}{4 \pi^2} \left(  3 [d \beta_1] \cup [d \beta_1 ] + 3  [d \beta_2] \cup [d \beta_2] \right)  , 
\end{align*}
so these three bundles are all topologically different.

\bibliography{refs}
%===============================================================================

\end{document}